\documentclass[12pt,reqno] {amsart}
\usepackage{xcolor}
\usepackage{eucal}
\usepackage{amsthm}
\usepackage{amsfonts}
\usepackage{amsmath}
\usepackage{amssymb}
\usepackage{mathrsfs}
\usepackage{epsfig}
\usepackage{float}
\usepackage{MnSymbol} 
\newcommand{\ncom}{\newcommand}

\ncom{\ul}{\underline}
\ncom{\ol}{\overline}
\ncom{\bq}{\begin{equation}}
\ncom{\eq}{\end{equation}}
\ncom{\beqn}{\begin{eqnarray*}}
\ncom{\eeqn}{\end{eqnarray*}}
\ncom{\beq}{\begin{eqnarray}}
\ncom{\eeq}{\end{eqnarray}}
\ncom{\nno}{\nonumber}
\ncom{\rar}{\rightarrow}
\ncom{\Rar}{\Rightarrow}
\ncom{\noin}{\noindent}
\ncom{\bc}{\begin{centre}}
\ncom{\ec}{\end{centre}}
\ncom{\sz}{\scriptsize}
\ncom{\rf}{\ref}
\ncom{\sgm}{\sigma}
\ncom{\Sgm}{\Sigma}
\ncom{\dt}{\delta}
\ncom{\Dt}{Delta}
\ncom{\s}{\underline{s}}
\ncom{\lmd}{\lambda}
\ncom{\Lmd}{\Lambda}
\ncom{\eps}{\epsilon}
\ncom{\pcc}{\stackrel{P}{>}}
\ncom{\dist}{{\rm\,dist}}
\ncom{\sspan}{{\rm\,span}}
\ncom{\re}{{\rm Re\,}}
\ncom{\im}{{\rm Im\,}}
\ncom{\sgn}{{\rm sgn\,}}
\ncom{\ba}{\begin{array}}
\ncom{\ea}{\end{array}}
\ncom{\eop}{\hfill{{\rule{2.5mm}{2.5mm}}}}
\ncom{\eoe}{\hfill{{\rule{1.5mm}{1.5mm}}}}
\ncom{\eof}{\hfill{{\rule{1.5mm}{1.5mm}}}}
\ncom{\hone}{\mbox{\hspace{1em}}}
\ncom{\htwo}{\mbox{\hspace{2em}}}
\ncom{\hthree}{\mbox{\hspace{3em}}}
\ncom{\hfour}{\mbox{\hspace{4em}}}
\ncom{\hsev}{\mbox{\hspace{7em}}}
\ncom{\vone}{\vskip 2ex}
\ncom{\cH}{{\mathcal H}}
\ncom{\vtwo}{\vskip 4ex}
\ncom{\vonee}{\vskip 1.5ex}
\ncom{\vthree}{\vskip 6ex}
\ncom{\vfour}{\vspace*{8ex}}
\ncom{\norm}{\|\;\;\|}
\ncom{\integ}[4]{\int_{#1}^{#2}\,{#3}\,d{#4}}
\ncom{\inp}[2]{\langle{#1},\,{#2} \rangle}
\ncom{\Inp}[2]{\left\langle{#1},\,{#2} \right\rangle}
\ncom{\vspan}[1]{{{\rm\,span}\#1 \}}}
\ncom{\dm}[1]{\displaystyle {#1}}
\ncom{\Hom}{\operatorname{Hom}}
\ncom{\Hol}{\operatorname{Hol}}
\ncom{\Ps}{\mathcal P_{\underline{s}}}
\ncom{\hl}{\mathcal H}

\ncom{\defin} {\overset {\text {\rm def} }{=}}

\newtheorem{theorem}{\bf Theorem}[section]
\newtheorem{example}[theorem]{\bf Example}%[section]
\newtheorem{proposition}[theorem]{\bf Proposition}%[section]
\newtheorem{corollary}[theorem]{\bf Corollary}%[section]
\newtheorem{lemma}[theorem]{\bf Lemma}%[section]
\newtheorem{remark}[theorem]{\bf Remark}%[section]
\newtheorem{definition}[theorem]{\bf Definition}%[section]

\def \N{\mathbb{N}}

\def \s{\underline{s}}

\renewcommand{\epsilon}{\varepsilon}
\renewcommand{\kappa}{\varkappa}

\begin{document}

\title[Cartan isometry]{Cartan Isometries and Toeplitz Operators on Cartan domains}

\author[S. Kumar ]{Surjit Kumar}
\address[S. Kumar]{Department of Mathematics, Indian Institute of Technology Madras, Chennai 600036, India} 
\email{surjit@iitm.ac.in}
\author[M. K. Mal]{Milan Kumar Mal}
\address[M. K. Mal]{Department of Mathematics, Indian Institute of Technology Madras, Chennai 600036, India}
 \email{ma21d018@smail.iitm.ac.in; milanmal1702@gmail.com }
\author[P. Pramanick]{Paramita Pramanick}
\address[P. Pramanick]{Stastistics and Mathematics Unit, Indian Statistical Institute Kolkata, Kolkata 700108, India}
\email{paramitapramanick@gmail.com}

\thanks{Support for the work of S. Kumar was provided in the form of MATRICS grant (Ref No. MTR/2022/000457) of Anusandhan National Research Foundation (ANRF). 
Support for the work of M. K. Mal was provided in the form of a Prime Minister's Research Fellowship (PMRF / 2502827). 
Support for the work of P. Pramanick was provided by the Department of Science and Technology (DST) in the form of the Inspire Faculty Fellowship (Ref No. DST/INSPIRE/04/2023/001530).}


\date{}

\begin{abstract}
We provide a description of the Shilov boundary of the classical Cartan domain in terms of Jordan triple determinant. As a consequence, we obtained an intrinsic characterization of Cartan isometries. Further, we obtain (i) invariance of Cartan isometries under the action of the biholomorphic automorphism group, and (ii) a Brown-Halmos type condition for Toeplitz operators on the Cartan domain. Also, we show that the zero operator is the only compact Toeplitz operator. Finally, we study the $\boldsymbol T$-Toeplitz operators and reflexivity of a Cartan isometry $\boldsymbol T.$ 
\end{abstract}

\renewcommand{\thefootnote}{}

\maketitle
%%%%%%%%%%%%%%%%%%%%%%%%%%%%%%%%%%%%%%%%%%%%%%%%%%%%%%%%%%%%%%%%%%%%%%%%%%%%%%%%%%%%%%%%%%%
%%%%%%%%%%%%%%%%%%%%%%%%%%%%%%%%%%%%%%%%%%%%%%%%%%%%%%%%%%%%%%%%%%%%%%%%%%%%%%%%%%%%%%%%%%%

\section{Introduction}
Let $\Omega$ be an irreducible bounded symmetric domain in $\mathbb C^d$ of type $(r, a, b)$ in its Harish-Chandra realization. Let $G$ be the connected component of the identity in $\rm{Aut}(\Omega)$, the biholomorphic automorphism group of $\Omega$. Let $\mathbb{K}=\{g \in G : g(0)=0\}$ be a maximal compact subgroup of $G$.
By \cite[Proposition 2, pp. 67]{RN},  $\mathbb K$ is a subgroup of linear automorphism in $G.$  Thus, the action of $\mathbb{K}$ on $\Omega$ extends naturally to $\mathbb C^d$ given by
 \[k\cdot z:=\big(k_1(z), \ldots, k_d(z)\big),\; \; k\in \mathbb{K} \mbox{~and~} z\in \mathbb C^d.\]
The complete classification of irreducible bounded symmetric domains is given by \'E. Cartan \cite{Ca}. There are six types of irreducible bounded symmetric domains up to biholomorphic equivalence, the first four of which are called the {\it classical Cartan domains}. The other two types of domains are known as {\it exceptional domains}.
The numerical invariants $(r, a, b)$ determine the domain $\Omega $ uniquely up to biholomorphic equivalence. The dimension $d$ is given by the relation $d=r+\frac{a}{2}r(r-1)+rb.$ We refer to \cite{Loos, Arazy} for more details of bounded symmetric domains. 

Every irreducible bounded symmetric domain $\Omega$ of rank $r$ can be realized as an open unit ball of a {\it Cartan factor} $Z \approx \mathbb C^d$. There exists a frame $e_1, \ldots, e_r $ of pairwise orthogonal minimal tripotents, known as a Jordan frame, such that for each $z \in \mathbb C^d $ has a {\it polar decomposition} \[z= k \cdot \sum_{j=1}^r t_j e_j, \,\, t_1 \geq \ldots \geq t_r \geq 0, \, k \in \mathbb{K}.\] 
The numbers $t_1,\ldots, t_r$ are singular numbers of $z$ determined uniquely, but $k$ need not be. Furthermore, $z$ belongs to $\Omega$, $\partial\Omega$ or $\Omega^c $ if and only if $t_1 < 1, t_1 =1,$ or $t_1\geq 1$, respectively.
The domain $\Omega$ is the open unit ball $\{ z \in \mathbb C^d: \|z\| <1\},$ where $\|z\|$ is the largest singular number of $z$. The {\it Shilov boundary} $S_\Omega$ of $\Omega$ is given by $\{ z \in \mathbb C^d: z= k \cdot e, k \in \mathbb K\}$, where $e=e_1+\cdots+e_r$ is a maximal tripotent. There exists a unique $\mathbb K$ invariant probability measure supported on $S_{\Omega},$ which we denote by $dz.$  The Hardy space $H^2(S_{\Omega})$ is the space of all holomorphic functions $f:\Omega \to \mathbb C$ such that
\[\|f\|^2:=\sup_{0<t<1} \int_{S_{\Omega}} |f(t z)|^2 dz < \infty.\]

Throughout the paper, $\Omega$ will denote the classical Cartan domain. Let ${\mathbb N}$ denote the set of all nonnegative integers. For a set $X$ and positive integer $d$,  $X^d$ stands for the $d$-fold Cartesian product of $X$.
 Let $\mathcal{B}({\mathcal H})$ denote the unital Banach algebra of bounded linear operators on a complex separable Hilbert space $\mathcal H.$ A $d$-tuple of commuting bounded linear operators $T_1,\ldots, T_d$ defined on $\mathcal{H}$ is denoted by $\boldsymbol{T}=(T_1,\ldots, T_d)$. The notations $\sigma(\boldsymbol T)$, $\sigma_p(\boldsymbol T)$ and $\sigma_{\pi}(\boldsymbol T)$ are reserved for the Taylor joint spectrum, the joint point spectrum, and the joint approximate point spectrum of $\boldsymbol T$, respectively.

The group $\mathbb K$ acts naturally on the space of analytic polynomials  $\mathcal P(Z)$ on $Z$ by composition, that is, $(k\cdot p)(z)=p(k^{-1} \cdot z)$, $k\in \mathbb K,~ p\in \mathcal P(Z).$ 
Under this action, $\mathcal P(Z)$ decomposes into irreducible, mutually $\mathbb{K}$-inequivalent subspaces $\mathcal P_{\s}$ such that \[ \mathcal{P}(Z)=\sum_{\s} \mathcal{P}_{\s },\] where $\s=(s_1, \ldots , s_r)\in {\mathbb N}^r, s_1 \geq \ldots \geq s_r\geq 0$ is known as {\it signature}. The set of all signatures is denoted by $\vec{\mathbb N}^r.$ The above decomposition of $\mathcal{P}(Z)$ is called the {\it Peter-Weyl decomposition} in \cite[Section 3]{Hu} (see also \cite[page 21]{Arazy}). The reproducing kernel $K_{\s}$ of the space $\mathcal{P}_{\s }$ with respect to the Fischer-Fock inner product
\[\langle p,q\rangle_{F}:=\frac{1}{\pi^d}\int_{\mathbb{C}^d}p(z)\overline{q(z)}e^{-| z|^2}dm(z)\]
is $\mathbb K$-invariant, where $dm(z)$ is the Lebesgue measure. For $\s \in \vec{\mathbb N}^r,$ let $\{\psi_\alpha^{\s}(z)\}_{\alpha=1}^{d_{\s}}$ denote an orthonormal basis of $\mathcal P_{\s}$  with respect to the Fischer-Fock inner product, where $d_{\s}$ is the dimension of $\mathcal P_{\s}.$ The reproducing kernel $K_{\s}$ is given by
\beq \label{ortho} K_{\s}(z,w)= \sum_{\alpha=1}^{d_{\s}} \psi_\alpha^{\s}(z) \overline{\psi_\alpha^{\s}(w)}. \eeq

The Bergman space over $\Omega$ consists of square integrable holomorphic functions on $\Omega$ with respect to the Lebesgue measure $dm(z)$. It is known that the Bergman space is a reproducing kernel Hilbert space determined by the (Bergman) kernel $B(z,w)=\Delta (z,w)^{-p},$ where $p=2+a(r-1)+b,$ is called the {\it genus} of the domain $\Omega$ (see \cite[Theorem 2.9.8 ]{upmeier}). 
Here $\Delta(z,w)$ is the {\it Jordan triple determinant}, a $\mathbb K$-invariant sesqui-analytic polynomial on $\Omega \times \Omega$ uniquely determined by the property $\Delta(z,z)=\prod_{i=1}^r(1-t^2_i).$  
The Jordan triple determinant has the following decomposition \[\Delta(z,w)=\sum_{\ell =0}^r (-1)^{\ell} \Delta^{(\ell)}(z,w),\;\; z,w \in \Omega,\] where $(\ell):= (1,\ldots,1,0,\ldots,0) \in \Vec{\mathbb{N}^r}$ denotes signature with first $\ell$ many ones. The sesqui-analytic polynomials $\Delta^{(\ell)}(z,w)$ are homogeneous of bi-degree $(\ell,\ell)$ given by
\beqn \Delta^{(\ell)} (z,w)= (-1)^{\ell} (-1)_{(\ell)} K_{(\ell)}(z,w),\eeqn
where  $(x)_{\s}$ denotes the {\it generalized Pochhammer symbol}
\[(x)_{\s}:= \prod_{j=1}^{r} \left( x-\frac{a}{2}(j-1)\right)_{s_{j}}= \prod_{j=1}^{r}\prod_{l=1}^{s_{j}} \left(x-\frac{a}{2}(j-1)+l-1\right).\]

The sesqui-analytic function $K^{(\nu)}:\Omega \times \Omega \to \mathbb C$ defined by
\beq \label{FK formula} K^{(\nu)} (z, w):= \Delta(z,w)^{-\nu}= \sum_{\s} (\nu)_{\s}K_{\s}( z, w),\;\; z, w \in \Omega,\eeq is non negative definite, whenever $\nu  \in \left \{0, \ldots, \frac{a}{2}(r-1)\right \}\cup \left( \frac{a}{2}(r-1), \infty\right)$ (see \cite[Corollary 5.2]{FK} ).
The second equality in the equation \eqref{FK formula} is the Faraut-Kor\'anyi formula \cite[Theorem 3.8]{FK}. The non negative definite kernel $K^{(\nu)}$ gives rise to a reproducing kernel Hilbert space $\mathcal H^{(\nu)}(\Omega)$ of holomorphic functions on $\Omega,$ known as the {\it weighted Bergman space}. 
The Hardy space $H^2(S_{\Omega})$ on $S_\Omega$ corresponds to the parameter $\nu=\tfrac{d}{r}.$ 

 A commuting $d$-tuple $\boldsymbol{T}$ defined on $\mathcal H$ is said to be {\it subnormal} if there exists a commuting $d$-tuple $\boldsymbol N=(N_1,\ldots, N_d)$ of normal operators defined on a Hilbert space $\mathcal K \supseteq \mathcal H$ such that $N_i\mathcal H \subseteq \mathcal H, \; N_i|_{\mathcal H}=T_i$ for all $i=1,\ldots,d.$ The commuting normal extension $\boldsymbol N$ of subnormal tuple $\boldsymbol T$ is {\it minimal} if $\mathcal K= \bigvee \{\boldsymbol{N}^{* \alpha}h: h \in \mathcal{H}, \alpha \in \mathbb{N}^d\}$, which is unique up to unitary equivalence.
 
In \cite{At}, the notion of Cartan isometry is introduced and is defined as follows:
\begin{definition}\label{defn}
A commuting $d$-tuple $\boldsymbol{T}$ is said to be a {\it Cartan isometry} if it is subnormal and the Taylor joint spectrum $\sigma(\boldsymbol{N})$ of its minimal normal extension $\boldsymbol{N}$ is contained in the Shilov boundary $S_\Omega$ of $\Omega$.    
\end{definition}

When $\Omega$ is the open Euclidean unit ball $\mathbb B_d$, by \cite[Proposition 2]{At2}, Cartan isometry is simply a spherical isometry. A well-known example of a Cartan isometry is the Szeg\"o shift, the $d$-tuple $\boldsymbol T_{z}=(T_{z_1},\ldots,T_{z_d})$ of operators of multiplication by the coordinate functions $z_1,\ldots,z_d$ on the Hardy space $H^2(S_{\Omega}).$
 Athavale provided a characterization of Cartan isometry addressing each classical Cartan domain individually in \cite{At}.

In Section 2, we give another characterization of a Cartan isometry as follows: A commuting $d$-tuple $\boldsymbol{T}$ is a Cartan isometry if and only if  $\binom{r}{\ell} I_{\mathcal{H}} - \Delta^{(\ell)}(z,w)(\boldsymbol{T},\boldsymbol{T}^*)=0$ for all $1\leq \ell \leq r.$ 
This characterization, expressed in terms of Jordan triple determinant, is consistent for all classical Cartan domains. Here, the key idea is to recognize that the Shilov boundary $S_{\Omega}$ of the domain $\Omega$ is given by  
 $\left\{z \in \mathbb C^d: \Delta^{(\ell)}(z,z)= \binom{r}{\ell} \text{ for all } 1\leq \ell \leq r\right \}$ (see Lemma \ref{Shilov}).
 A similar description may be obtained for other boundary components of the domain $\Omega$. In the Appendix section, we provide such a description of type-I domain (see Theorem \ref{other-boundary}). However, in the present work, we restrict ourselves to the Shilov boundary, and thus, we will discuss only Cartan isometries.
We show that the Taylor joint spectrum $\sigma(\boldsymbol T)$ of a Cartan isometry $\boldsymbol T$ is contained in $\bar{\Omega}.$ Furthermore, we conclude that for any $\phi \in \rm {Aut}(\Omega),$ $\phi(\boldsymbol T)$ is a Cartan isometry whenever $\boldsymbol T$ is a Cartan isometry.

Recall that a bounded linear operator on $H^2(\mathbb D)$ defined as $T_\phi f:= P_{H^2(\mathbb{D})}(\phi f)$ is called a {\it Toeplitz operator}, where $\phi \in L^\infty $, $f \in H^2(\mathbb D)$ and $P_{H^2(\mathbb D)}$ denotes the orthogonal projection of $L^2$ onto $H^2(\mathbb D)$. The study of Toeplitz operators on function spaces has a rich history, beginning with the pioneering work of Brown and Halmos in \cite{Brown-Halmos}. Since then, extensive explorations have been made in the same direction for the domains such as polydisc, open unit ball, symmetrized bidisc, and Hartogs triangle; for instance, see \cite{Ding-Sun-Zheng, Ding, MSS, DJ, Tiret, Subham-Par, Ro-Go}. 
A study of Toeplitz operators on bounded symmetric domains appeared in \cite{upmeier1}. A structure theory of $C^*$-algebra of Toeplitz operators with continuous symbol defined on $S_{\Omega}$ was developed in \cite{upmeier2}. For more details, we refer to \cite{upmeier}.

In Section 3, we provide a Brown-Halmos type characterization of Toeplitz operators on the Hardy space $H^2(S_{\Omega})$ (see Theorem \ref{toeplitz}).  Also, we show that the zero operator is the only compact Toeplitz operator on $H^2(S_{\Omega}).$ Furthermore, we define a notion of $\boldsymbol T$-Toeplitz operators for a Cartan isometry $\boldsymbol T,$ and illustrate that this notion of $\boldsymbol T$-Toeplitz operators is consistent with the definition of $\boldsymbol T$-Toeplitz operator for an $A$-isometry defined in \cite{DEE}. Analogous to spherical isometry, we show that Cartan isometry is reflexive. We conclude this section with a description of dual Toeplitz operators on the Cartan domain.  
%%%%%%%%%%%%%%%%%%%%%%%%%

\section{A Characterization of Cartan Isometry}

The Jordan triple determinant associated with the irreducible bounded symmetric domain $\Omega$ has the following decomposition 
\[\Delta(z,w)=\sum_{\ell =0}^r (-1)^{\ell} \Delta^{(\ell)}(z,w),\;\; z,w \in \Omega.\] 
Here, the sesqui-analytic polynomial $\Delta^{(\ell)}(z,w),$ for each $\ell=1,\ldots, r,$ is homogeneous of bi-degree $(\ell,\ell)$  given by
\beq \label{delk} \Delta^{(\ell)} (z,w) \overset {\eqref{FK formula}}= (-1)^{\ell} (-1)_{(\ell)} K_{(\ell)}(z,w) \overset {\eqref{ortho}}= \prod_{j=1}^{\ell} (1+\frac{a}{2}(j-1))\sum_\alpha \psi_\alpha^{(\ell)}(z) \overline{\psi_\alpha^{(\ell)}(w)} .\eeq
For $\ell=1$, $(\ell)=(1,0,\ldots,0)$ and $\Delta^{(1)}(z,w)=\inp{z}{w}$ (see Section 3 of \cite{Up}).

In this section, we will provide an intrinsic characterization of Cartan isometry. We begin with the following lemma that describes the Shilov boundary of the domain $\Omega.$ 
Although it may be known to experts, for the sake of completeness, we include a proof. 
\begin{lemma} \label{Shilov}
    The Shilov boundary $S_{\Omega}$ of $\Omega$  is given by \[ \left\{z \in \mathbb C^d: \Delta^{(\ell)}(z,z)= \binom{r}{\ell} \text{ for all } 1\leq \ell \leq r\right \}.\] 
\end{lemma}
\begin{proof}
   
For any $z\in \mathbb C^d$, consider the polar decomposition $z= k\cdot \sum_{j=1}^r t_j e_j, \;t_1 \geq t_2\geq \ldots \geq t_r \geq 0$, for some $k \in \mathbb{K}$. Note that any $z \in S_{\Omega}$ is of the form $z=k\cdot e$ for some $k \in \mathbb K.$ Thus, it suffices to show that $\Delta^{(\ell)}(z,z)=\binom{r}{\ell}$ for all $1 \leq \ell \leq r$ if and only if $t_j=1$ for each $j=1,\ldots,r$ in the polar decomposition of $z.$
For any $\s \in \vec{\mathbb N}^r,$ we get
\[K_{\s}(z,z)=K_{\s}\left( k\cdot \sum_{j=1}^r t_j e_j, k\cdot \sum_{j=1}^r t_j e_j\right)=K_{\s}\left(\sum_{j=1}^r t_j e_j,\sum_{j=1}^r t_j e_j\right)=K_{\s}\left(\sum_{j=1}^r t_j^2 e_j, e\right),\] where the last equality follows from \cite[Lemma 3.2]{FK}. 
By \eqref{delk}, we have
\beqn \Delta(z,z)=\sum_{\ell=0}^r(-1)^{\ell} \Delta^{(\ell)}(\sum_{j=1}^r t_j^2 e_j, e)=\prod_{j=1}^r(1-t_j^2) =\sum_{\ell=0}^r (-1)^{\ell}\sum_{1 \leq j_1 <\ldots < j_{\ell}\leq r} t_{j_1}^2 \cdots t_{j_{\ell}}^2.\eeqn 
Thus, we get
\[\Delta^{(\ell)}(z,z)=\Delta^{(\ell)}(\sum_{j=1}^r t_j^2 e_j, e)=\sum_{1 \leq j_1 <\ldots < j_{\ell}\leq r} t_{j_1}^2 \cdots t_{j_{\ell}}^2, \;\; 1 \leq \ell \leq r.\]
 
 It now follows that $\Delta^{(\ell)}(z,z)= \binom{r}{\ell}$ for each $\ell=1,\ldots,r$ if and only if 
\beq \label{qwerty} \sum_{1 \leq j_1  <\ldots < j_{\ell}\leq r} t_{j_1}^2  \cdots t_{j_{\ell}}^2 =\binom{r}{\ell},\; \; 1\leq \ell \leq r. \eeq  
If  $\Delta^{(\ell)}(z,z)=\binom{r}{\ell}$, then $\Delta(z,z)= \prod_{j=1}^r (1-t_j^2)=0.$ Hence, there is at least one $j \in \{1, \ldots, r\}$ such that $t_j=1.$ Without loss of generality, let $t_1=1.$ Then the equation \eqref{qwerty} reduces to 
\[\sum_{2 \leq j_1  <\ldots < j_{\ell}\leq r} t_{j_1}^2  \cdots t_{j_{\ell}}^2 =\binom{r-1}{\ell},\; \; 1\leq \ell \leq r-1.\] Thus inductively, we get $t_j=1$ for all $1 \leq j \leq r$. This completes the proof. \end{proof}

Consider $D=\Omega_1 \times \cdots \times \Omega_k \subseteq \mathbb C^{d_1}\times \cdots \times \mathbb C^{d_k} \approx \mathbb C^{d_1+\cdots+d_k}$ to be a standard Cartan domain, where each $\Omega_i$ is a classical Cartan domain with rank $r_i$. Let $z_i=(z_{i,1},\ldots,z_{i,d_i}) \in \mathbb C^{d_i}$ then  $z= (z_1;\ldots;z_k)=(z_{1,1}, \ldots, z_{1,d_1};z_{2,1},\ldots,z_{2,d_2};\ldots, z_{k,d_k}) \in \mathbb C^d,$ where $d=d_1+\cdots+d_k$. It is known that the Shilov boundary $S_D$ of the domain $D$ is given by $S_D=S_{\Omega_1}\times \cdots \times S_{\Omega_k}.$ For each $1 \leq i \leq k$ and $1 \leq \ell_i \leq r_i$, let $\Delta_{\Omega_i}^{(\ell_i)}(z_i,w_i)$ be the sesqui-analytic homogeneous polynomials of bi-degree $(\ell_i,\ell_i)$ associated with the classical Cartan domain domain $\Omega_i$ given by the equation \eqref{delk}. A direct generalization of Lemma \ref{Shilov} gives a description of the Shilov boundary of $D$.

\begin{corollary} \label{S_D}
    The Shilov boundary $S_D$ of the domain $D$ is given by
    \[\left\{ z=(z_1;\ldots;z_k) \in \mathbb C^d: \Delta_{\Omega_i}^{(\ell_i)}(z_i,z_i)=\binom{r_i}{\ell_i}\; \; \text{for all}\;\; 1\leq \ell_i \leq r_i, \; 1 \leq i \leq k\right\}.\]
\end{corollary}

Note that the topological boundary $\partial\Omega$ of a classical Cartan domain $\Omega$ is the disjoint union of $r$ number of $G$-orbits, that is, $\partial \Omega= \bigcupdot_{j=1}^r \partial\Omega_j,$ where $\partial\Omega_j=\{g(\sum_{i=1}^je_i): g \in G\}$ for all $1 \leq j \leq r$. The Shilov boundary $S_{\Omega}$ is the $G$-orbit $\partial\Omega_r,$ which is also an irreducible $\mathbb K$-orbit (See \cite[page 18]{Arazy}). As we obtained a description of the Shilov boundary in Lemma \ref{Shilov}, a similar characterization can be deduced for the other boundary components. One such description for other boundary components of type-I domain is given below in the Appendix (see  Theorem \ref{other-boundary}). One may define a possible notion of isometries associated with other boundary components, as the Cartan isometry associated with the Shilov boundary is defined in Definition \ref{defn}, which we intend to investigate in our future research. However, in this article, we will focus exclusively on the Cartan isometry.

For any two analytic polynomials $p,q$ in $d$-variables, the hereditary functional calculus is defined as \[p(z)\overline{q(w)}(\boldsymbol{T},\boldsymbol{T}^*):=q(\boldsymbol{T})^* p(\boldsymbol{T}).\] In a private communication with Sameer Chavan, we found out that the next result has been recorded in one of his unpublished works. For the sake of the reader's convenience, we include the proof.

\begin{lemma}\label{lemmaSC}
    Let $\boldsymbol{T}$ be a subnormal $d$-tuple defined on $\mathcal H$ and $\boldsymbol{N}$ denote its minimal normal extension on $\mathcal{K}\supseteq \mathcal H.$ Let $p_i, q_i$ be $2n$ analytic polynomials in $d$-variables. If \[ \sum_{i=1}^n p_i(z)\overline{q_i(w)}(\boldsymbol{T},\boldsymbol{T}^*)= I_{\mathcal{H}},\; \text{ then } \;
    \sum_{i=1}^n p_i(z)\overline{q_i(w)}(\boldsymbol{N},\boldsymbol{N}^*)=  I_{\mathcal{K}}.\]
\end{lemma}

\begin{proof}
For any $h, g \in \mathcal{H}$ and $\alpha, \beta \in \mathbb{N}^d$, it follows by using repetition of Fuglede-Putnam's theorem,
\beqn \Inp{\sum_{i=1}^n p_i(z)\overline{q_i (w)}(\boldsymbol{N}, \boldsymbol{N}^*){\boldsymbol{N}}^{*\alpha}h}{{\boldsymbol{N}}^{*\beta}g}_{\mathcal{K}}
           &=& \sum_{i=1}^n \Inp{\boldsymbol{N}^\beta p_i (\boldsymbol{N}) h}{\boldsymbol{N}^\alpha q_i (\boldsymbol{N})g}_{\mathcal{K}} \\
          &=& \sum_{i=1}^n \Inp{\boldsymbol{T}^\beta p_i (\boldsymbol{T}) h}{\boldsymbol{T}^\alpha q_i (\boldsymbol{T})g}_{\mathcal{H}}\\
           &=& \big \langle\sum_{i=1}^n p_i(z)\overline{q_i (w)}(\boldsymbol{T},\boldsymbol{T}^*)\boldsymbol{T}^\beta h, \boldsymbol{T}^\alpha g\big \rangle_\mathcal{H}\\
         &=&\inp{\boldsymbol{T}^\beta h}{\boldsymbol{T}^\alpha g}_{\mathcal{H}}\\
    &=&\inp{\boldsymbol{N}^\beta h}{\boldsymbol{N}^\alpha g}_{\mathcal{K}}
        = \inp{\boldsymbol{N}^{*\alpha} h}{\boldsymbol{N}^{*\beta} g}_{\mathcal{K}}. 
         \eeqn
The conclusion follows from the minimality of the normal extension.
\end{proof}

In the following result, we provide a characterization of Cartan isometries.
\begin{theorem}\label{characterization}
Let $\boldsymbol{T}$ be a commuting $d$-tuple of bounded linear operators defined on $\mathcal{H}$. Then $\boldsymbol{T}$ is a Cartan isometry if and only if  \beq \label{Cartaniso}\Delta^{(\ell)}(z,w)(\boldsymbol{T},\boldsymbol{T}^*)=\binom{r}{\ell} I_{\mathcal{H}}\; \mbox{~for~all~} 1 \leq \ell \leq r.\eeq
\end{theorem}

\begin{proof}
    Let $\boldsymbol{T}$ be a Cartan isometry, and let $\boldsymbol N$ be its minimal normal extension. Let $\{\psi_{\alpha}^{(\ell)}\}$ be an orthonormal basis of $\mathcal{P}_{(\ell)}$ with respect to the Fischer-Fock inner product.  Then for any $ 1 \leq \ell \leq r$ and  $h \in \mathcal H$, we have 
\beqn
     \Inp{\Delta^{(\ell)}(z,w)(\boldsymbol{T},\boldsymbol{T}^*)h}{h}_{\mathcal{H}}
    &\overset{\eqref{delk}}=& \prod_{j=1}^{\ell} (1+\frac{a}{2}(j-1))\big \langle\sum_{\alpha}\psi_{\alpha}^{(\ell)}(z) \overline{\psi_{\alpha}^{(\ell)}(w)}(\boldsymbol{T},\boldsymbol{T}^*)h, h\big \rangle_{\mathcal{H}}\\
    &=&\prod_{j=1}^{\ell} (1+\frac{a}{2}(j-1))\sum_{\alpha}\Inp{ {\psi_{\alpha}^{(\ell)}(\boldsymbol{N})}h}{\psi_{\alpha}^{(\ell)}(\boldsymbol{N})h}_{\mathcal{K}}\\
    &=&\int_{\sigma(\boldsymbol N)}\prod_{j=1}^{\ell} (1+\frac{a}{2}(j-1))\sum_{\alpha}\psi_{\alpha}^{(\ell)}(z) \overline{\psi_{\alpha}^{(\ell)}(z)} d\inp{E(z)h}{h}_{\mathcal{K}}\\
    &=&\int_{\sigma(\boldsymbol N) }\Delta^{(\ell)}(z,z) d\inp{E(z)h}{h}_{\mathcal{K}},\eeqn
    where $E$ is the spectral measure for $\boldsymbol N.$
Since $\sigma(\boldsymbol N) \subseteq S_{\Omega},$ by Lemma \ref{Shilov}, we get
\beqn  
   \Inp{\Delta^{(\ell)}(z,w)(\boldsymbol{T},\boldsymbol{T}^*)h}{h}_{\mathcal{H}} = \binom{r}{\ell} \inp{h}{h}_{\mathcal{H}}.
\eeqn
Conversely, assume that for each $1 \leq \ell \leq r,$  $\Delta^{(\ell)}(z,w)(\boldsymbol{T},\boldsymbol{T}^*)= \binom{r}{\ell} I_{\mathcal{H}}$. In particular for $\ell =1$, $\Delta^{(1)}(z,w)(\boldsymbol{T},\boldsymbol{T}^*)= \sum_{i=1}^d T_i^* T_i = r I_{\mathcal{H}}$.
By \cite[Proposition 2]{At2}, $\boldsymbol{T}$ is subnormal. Let $\boldsymbol{N}=(N_1,\ldots N_d)$ be the minimal normal extension of $\boldsymbol{T}$ on $\mathcal{K}\supseteq \mathcal H$. To show that $\boldsymbol{T}$ is a Cartan isometry, it suffices to show that Taylor joint spectrum $\sigma(\boldsymbol{N})$ is contained in $S_\Omega$.
By Lemma \ref{lemmaSC},  $\Delta^{(\ell)}(z,w)(\boldsymbol{N},\boldsymbol{N}^*)= \binom{r}{\ell} I_{\mathcal{K}}$ for all $1 \leq \ell \leq r.$ By using functional calculus for $\boldsymbol N,$  $\Delta^{(\ell)}(z,z)=\binom{r}{\ell}$ for all $z \in \sigma(\boldsymbol{N}).$ Now the conclusion follows from Lemma \ref{Shilov}.

\end{proof}

 Let $\boldsymbol{T}=(T_{1,1},\ldots,T_{1,d_1}; T_{2,1},\ldots,T_{2,d_2};\ldots,T_{k,d_k})$ be a commuting $d$-tuple of bounded linear operators defined on $\mathcal{H}$, where $d=d_1+\cdots+d_k.$ Set $\boldsymbol{T}_i=(T_{i,1},\ldots,T_{i,d_i}),$ then $\boldsymbol{T}=(\boldsymbol{T}_1;\ldots;\boldsymbol{T}_k).$
 The following result analogous to Theorem \ref{characterization} can be obtained by using Corollary \ref{S_D} for the standard Cartan domain $D=\Omega_1 \times \cdots \times \Omega_k \subseteq \mathbb C^{d_1}\times \cdots \times \mathbb C^{d_k} \approx \mathbb C^{d}.$
\begin{corollary}
 Let $\boldsymbol T$ be the $d$-tuple of operators specified as above. Then $\boldsymbol T$ is subnormal with the Taylor joint spectrum $\sigma(\boldsymbol{N})$ of its minimal normal extension $\boldsymbol{N}$ is contained in the Shilov boundary $S_D$ of the standard Cartan domain $D$ if and only if \[\Delta_{\Omega_i}^{(\ell_i)}(z,w)(\boldsymbol{T}_i,\boldsymbol{T}_i^*)=\binom{r_i}{\ell_i}I_{\mathcal{H}}\;\; \text{for all}\;\; 1\leq \ell_i \leq r_i,\; 1 \leq i \leq k.\]
\end{corollary}

For a Cartan isometry $\boldsymbol T,$ it is reasonable to expect that the Taylor joint spectrum $\sigma(\boldsymbol{T})$ of $\boldsymbol T$ is contained in $\bar{\Omega},$ the closure of the domain $\Omega.$ 

In order to prove this, we first establish a couple of lemmas. We begin by setting the following notation. The idea is motivated from \cite{Ch-Ze}.
 \[r_\Omega(\boldsymbol{T}):=\sup \{\|z\|: z \in \sigma(\boldsymbol{T})\} \; \text{and} \; r_\Omega^\pi (\boldsymbol{T}):=\sup \{\|z\| : z \in \sigma_\pi (\boldsymbol{T})\},\]
where  $\|z\|$ is the spectral norm of $z$ associated to the domain $\Omega$.

\begin{lemma}\label{specrad}
    For a commuting $d$-tuple $\boldsymbol{T}$, $r_\Omega(\boldsymbol{T})=r_\Omega^\pi (\boldsymbol{T})$.
\end{lemma}

\begin{proof}
    Clearly, $r_\Omega^\pi (\boldsymbol{T}) \leq r_\Omega(\boldsymbol{T})$. Let $z_0 \in \sigma(\boldsymbol{T})$ be such that $\|z_0\|= r_\Omega(\boldsymbol{T})$. By the Hahn-Banach theorem, there exists a linear functional $\phi :\mathbb C^d\to \mathbb{C}$ such that $\phi(z_0)= \|z_0\|$ and $|\phi(z)| \leq \|z\|$, for all $z \in \mathbb C^d$. 
    Let $A =\phi(\boldsymbol{T}),$ then $A \in \mathcal{B}(\mathcal{H})$ and therefore by polynomial mapping theorem, \[\sigma(A)=\{\phi(z): z \in \sigma(\boldsymbol{T})\} \quad \text{ and}\quad\sigma_\pi (A)= \{\phi(z): z \in \sigma_\pi(\boldsymbol{T})\}.\]
    For any $\lambda \in \sigma(A),$ $|\lambda|= |\phi(z)| \leq \|z\| \leq r_\Omega(\boldsymbol{T})$. 
    Thus $r(A) \leq r_\Omega(\boldsymbol{T}),$ where $r(A)$ is the spectral radius of $A.$ Similarly,
    $r_\pi (A) :=\sup\{|\lambda|: \lambda \in \sigma_{\pi}(A)\}\leq r_\Omega^\pi(\boldsymbol{T}).$
    Since $\phi(z_0)=\|z_0\|=r_\Omega(\boldsymbol{T})$, $r(A) = r_\Omega(\boldsymbol{T})$. Now, the inequality $r_\Omega(\boldsymbol{T}) \leq r_\Omega^\pi (\boldsymbol{T})$ follows from the fact that $r(A)=r_\pi(A)$.
\end{proof}

\begin{lemma}\label{appspecT}
	 If $\boldsymbol{T}$ is a Cartan isometry then the joint approximate point spectrum $\sigma_{\pi}(\boldsymbol{T})$ is contained in $S_{\Omega}.$ 
\end{lemma}

\begin{proof}
 Let $z \in \sigma_{\pi}(\boldsymbol{T})$ then there exists a sequence of unit vectors $x_k$ such that for all $\alpha \in \N^d,$  $\lim_{k \to \infty}(\boldsymbol T^\alpha- z^\alpha)x_k= 0$ (see the discussion at page 187 in \cite{GR}).
 So, for any polynomial  $p$  in $d$-variables,  $\lim_{k\to \infty}\big(p(\boldsymbol{T})-p(z)\big)x_k=0$. Since
\beqn
\big|\|p(\boldsymbol{T})x_k\|-|p(z)|\big|=\big|\|p(\boldsymbol{T})x_k\|-\|p(z)x_k\|\big| \leq \|p(\boldsymbol{T})x_k -p(z)x_k\| \to 0 \text{ as } k \to \infty, \eeqn
 $\|p(\boldsymbol{T})x_k\| \to |p(z)|$ as $k \to \infty$.
It follows from Theorem \ref{characterization} that
\beqn
		\binom{r}{\ell} & \overset{\ref{delk}}=& \prod_{j=1}^{\ell} (1+\frac{a}{2}(j-1))\big \langle\sum_{\alpha}\psi_{\alpha}^{(\ell)}(\boldsymbol{T})^*\psi_{\alpha}^{(\ell)}(\boldsymbol{T})x_k, x_k\big \rangle \\ &=& \prod_{j=1}^{\ell} (1+\frac{a}{2}(j-1))\sum_{\alpha}\|\psi_{\alpha}^{(\ell)}(\boldsymbol{T})x_k\|^2.
\eeqn
  Now letting $k \to \infty,$ we get  \[\Delta^{(\ell)}(z,z)=\prod_{j=1}^{\ell} (1+\frac{a}{2}(j-1))\sum_{\alpha}|\psi_{\alpha}^{(\ell)}(z)|^2= \binom{r}{\ell},\]
  for all $ 1 \leq \ell \leq r.$ Now the conclusion follows from Lemma \ref{Shilov}.
\end{proof}

\begin{proposition}\label{specT}
    If $\boldsymbol{T}$ is a Cartan isometry, then the Taylor joint spectrum $\sigma(\boldsymbol{T}) \subseteq \bar{\Omega}$.
\end{proposition}
\begin{proof}
Note that for any $z \in S_\Omega,$ the spectral norm $\|z\|=1.$ By Lemma \ref{appspecT}, $r_\Omega^\pi(\boldsymbol{T}) =1$. Since by Lemma \ref{specrad}, $r_\Omega(\boldsymbol{T}) =1$, which gives  $\|z\| \leq 1$, for any $z \in \sigma(\boldsymbol{T}),$ that is, $z \in \bar{\Omega}$. This completes the proof.
\end{proof}

Any biholomorphic automorphism $\phi \in \rm{Aut}(\Omega)$ extends to an analytic function in some neighbourhood of $\bar{\Omega}$ (see \cite[page 18]{Arazy}) and we denote this extension by $\phi$ again. Moreover, $\phi$ is a homeomorphism of $\bar{\Omega}$ which maps the Shilov boundary onto the Shilov boundary of the domain $\Omega.$ For more details the reader is referred to \cite[Section 2]{Arazy} (see also \cite[Theorem 2.4]{MM}). Let $\boldsymbol{T}$ be a Cartan isometry, and let $\boldsymbol N$ be its minimal normal extension. By Proposition \ref{specT}, $\sigma(\boldsymbol{T}) \subseteq \bar{\Omega}.$ One can define $\phi(\boldsymbol{T})$ using the Riesz-Dunford functional calculus. And $\phi(\boldsymbol{N})$ is defined via continuous functional calculus. 
It is easy to verify that $\phi(\boldsymbol{N})\mathcal H \subseteq \mathcal H$ and $\phi(\boldsymbol{N})|_\mathcal{H}= \phi(\boldsymbol{T})$.
%It is easy to verify that $\mathcal{H}$ is invariant under $\phi(\boldsymbol N)$ and $\phi(\boldsymbol{N})|_\mathcal{H}= \phi(\boldsymbol{T})$.
The following result shows that $\phi(\boldsymbol{T})$ is also a Cartan isometry. In case when $\Omega=\mathbb B_d$, an analogous result has been obtained for $m$-isometry in \cite[Theorem 4.3]{Gu}. 

 \begin{theorem}\label{invariance}
 	Let $\boldsymbol{T}$ be a Cartan isometry. Then for every biholomorphic automorphism $\phi \in \rm{Aut}(\Omega), \; \phi(\boldsymbol{T})$ is a Cartan isometry.
 \end{theorem}
 
 \begin{proof}
 	Let $\boldsymbol{N}$ be the minimal normal extension of $\boldsymbol{T} $ on $\mathcal{K} \supseteq \mathcal{H}$ and $E$ be the spectral measure for $\boldsymbol N$.
 	For any $1 \leq \ell \leq r$ and $h \in \mathcal{H}$,
 	
 	\beqn
& &\Inp{\Delta ^{(\ell)} (z,z)\big(\phi(\boldsymbol T), \phi(\boldsymbol T)^*\big)h}{h}_{\mathcal H} \\ &\overset{\eqref{delk}} = &
 	  \prod_{j=1}^{\ell} (1+\frac{a}{2}(j-1))\big \langle \sum_{\alpha}\psi_{\alpha}^{(\ell)}(\phi(\boldsymbol{T}))^*\psi_{\alpha}^{(\ell)}(\phi(\boldsymbol{T}))h,h\big \rangle_{\mathcal H}\\
    &=&\prod_{j=1}^{\ell} (1+\frac{a}{2}(j-1))\sum_{\alpha} \inp{\psi_{\alpha}^{(\ell)}(\phi(\boldsymbol{T})) h}{\psi_{\alpha}^{(\ell)}(\phi(\boldsymbol{T}))h}_{\mathcal{H}}\\
   &=& \prod_{j=1}^{\ell} (1+\frac{a}{2}(j-1))\sum_{\alpha} \Inp{\psi_{\alpha}^{(\ell)}(\phi(\boldsymbol{N})) h}{\psi_{\alpha}^{(\ell)}(\phi(\boldsymbol{N}))h}_{\mathcal{K}}\\
 	&=& \prod_{j=1}^{\ell} (1+\frac{a}{2}(j-1))\int_{\sigma(\boldsymbol N)} \sum_{\alpha}  \psi_{\alpha}^{(\ell)} ( \phi(z)) \overline{\psi_{\alpha}^{(\ell)} ( \phi(z)) } d\inp{E(z)h}{h}_{\mathcal K} \eeqn
    Since $\phi(S_{\Omega})=S_{\Omega}$ and $\sigma(\boldsymbol N) \subseteq S_{\Omega},$ Lemma \ref{Shilov} gives
  \[ \Inp{\Delta ^{(\ell)} (z,z)\big(\phi(\boldsymbol T), \phi(\boldsymbol T)^*\big)h}{h}_{\mathcal H}=\binom{r}{\ell}\int_{\sigma(\boldsymbol N)} d\inp{E(z)h}{h}=\binom{r}{\ell} \inp{h}{h}_{\mathcal{H}}.\]
  The conclusion follows from Theorem \ref{characterization}.
 \end{proof}
 
If the orbit of the action of the group $\rm{Aut}(\Omega)$ on $\boldsymbol T$ modulo unitary equivalence is a singleton set, then we say $\boldsymbol T$ is homogeneous. It is instinctive to wonder what the homogeneous Cartan isometries might be. We address this by extending it to a much broader class known as $\mathbb K$-homogeneous, where $\mathbb K$ is a maximal compact subgroup of $\rm{Aut}(\Omega).$
For any $k \in \mathbb K,$ there is a natural action $k\cdot \boldsymbol T$ of $\mathbb K$ on a $d$-tuple of commuting bounded linear operators $\boldsymbol T.$ The $d$-tuple $\boldsymbol T$ is said to be $\mathbb{K}$-homogeneous if $k\cdot \boldsymbol T$ and $\boldsymbol T$ are unitarily equivalent for all $k$ in $\mathbb{K}$. 
In the following result, we show that, up to a unitary equivalence, the only Cartan isometry in a certain class of $\mathbb{K}$-homogeneous operators and thus for homogeneous operators is the Szeg\"o shift.

\begin{theorem} 
Let $\boldsymbol T$ be a $\mathbb{K}$-homogeneous $d$-tuple such that the joint kernel $\bigcap_{i=1}^d \ker T^*_i$ of $\boldsymbol T^*$ is a one dimensional cyclic subspace for $\boldsymbol T$ and $\Omega \subseteq \sigma_p(\boldsymbol T^*).$ Then $\boldsymbol T$ is a Cartan isometry if and only if $\boldsymbol T$ is unitarily equivalent to the Szeg\"o shift.
\end{theorem}

\begin{proof} It follows from \cite[Theorem 2.3]{GKP} that $\boldsymbol T$ is unitarily equivalent to the $d$-tuple $\boldsymbol {M}_z=(M_{z_1}, \ldots, M_{z_d})$ of operators of multiplication by the coordinate functions $z_1,\ldots, z_d$ on a reproducing kernel Hilbert space $\mathcal{H}(K)$ determined by the $\mathbb K$-invariant kernel $K(z,w)=\sum_{\s}a_{\s}K_{\s}(z,w)$ with $a_0=1.$ To prove the result, it suffices to show that $\boldsymbol M_z$ is the Szeg\"o shift. Let $\boldsymbol M_z$ be a Cartan isometry. Since the inner product on $\mathcal{H}(K)$ is $\mathbb K$-invariant, the scalar spectral measure associated to $\boldsymbol{M}_z$ is $\mathbb K$-invariant and supported on the Shilov boundary  $S_{\Omega}$. The proof now follows from the fact that there is a unique $\mathbb K$-invariant probability measure supported on the $S_{\Omega}.$
\end{proof}
When $\mathbb K=\mathcal U(d)$, that is, in the $\mathcal U(d)$-homogeneous case, the above result was noted in \cite[Remark 2.4]{CY}.

We end this section by presenting examples of Cartan isometries that are not $\mathbb K$ homogeneous, in particular, which are not unitarily equivalent to the Szeg\"o shift.
 \begin{example}
 Let $\mu$ be a Borel probability measure (not necessarily $\mathbb K$-invariant) supported on $S_\Omega$ of $\Omega.$ Let $\boldsymbol M_{z}=(M_{z_1},\ldots, M_{z_d})$ denote the $d$-tuple of operators of multiplication by the coordinate functions $z_1,\ldots, z_d$ on $P^2(\mu),$ the norm closure of $\{p|_{S_{\Omega}}: p \in \mathcal P(Z)\}$ in $L^2 (S_{\Omega},\mu).$ Then $\boldsymbol M_{z}$ is a Cartan isometry. Here, we construct a measure that gives rise to a Cartan isometry that is not unitarily equivalent to the Szeg\"o shift.
Consider the elementary spherical function $\varphi_{\s},$ a unique $L$-invariant polynomial in $\mathcal{P}_{\s}$ such that $\varphi_{\s}(e)=1.$ Here $L$ is the stabilizer of $e$ in $\mathbb K.$  Note that $\|\varphi_{\s}\|_{H^2(S_\Omega)}^2= \frac{1}{d_{\s}}$. 
Let $\Omega=\{z \in \mathbb{M}_{2 \times m} (\mathbb C) : I_{2}- zz^* > 0\}$ be a type-I  Cartan domain, where $m\geq 2.$ Define a probability Borel measure $\nu$ supported on $S_\Omega$ as \[\nu(\Delta):= d_{(1)}\int_{\Delta}|\varphi_{(1)}(z)|^2 dz,\] for each Borel subset $\Delta$ of $S_\Omega$. Clearly, the $d$-tuple $\boldsymbol M_{z}$ of multiplication by coordinate functions on $P^2(\nu)$ is a Cartan isometry. Since the measure $\nu$ is not $\mathbb K$-invariant, $\boldsymbol M_{z}$ cannot be unitarily equivalent to the  Szeg\"o shift. 
 \end{example}

\section{Toeplitz Operators}
Brown and Halmos in \cite[Theorem 6]{Brown-Halmos} give a characterization for an operator $X \in \mathcal{B}(H^2 (\mathbb{D}))$ to be a Toeplitz operator if and only if $T_z^* XT_z=X$, where $T_z$ is the multiplication operator by the coordinate function $z$ on $H^2(\mathbb{D})$. A suitable Brown-Halmos type characterization of Toeplitz operators on the Hardy space over polydisc $\mathbb{D}^n$ and symmetrized bidisc $\mathbb G$ is obtained in \cite{MSS} and \cite{Tiret} respectively. Davie and Jewell in \cite[Theorem 2.6]{DJ}, obtained a Brown-Halmos type characterization of Toeplitz operators on the Hardy space $H^2(\mathbb{B}_d)$ of the unit ball $\mathbb B_d$. An operator $X \in \mathcal{B_d}(H^2 (\mathbb{B}_d))$ is Toeplitz if and only if $\sum_{i=1}^dT_{z_i}^* XT_{z_i}=X,$ where $\boldsymbol T_{z}=(T_{z_1}, \ldots, T_{z_d})$ is the tuple of multiplication operators by the coordinate functions $z_1, \ldots, z_d$ on $H^2(\mathbb{B}_d)$.

In this section, we discuss Brown-Halmos type characterization for the Toeplitz operators defined on the Hardy space $H^2 (S_\Omega)$ over an irreducible bounded symmetric domain $\Omega.$ Subsequently, we generalize the notion of $\boldsymbol T$-Toeplitz operator in the context of Cartan isometry $\boldsymbol T$. We begin with the definition of Toeplitz operators.
 The {\it Toeplitz operator} $T_{\phi}$ with symbol $\phi \in L^\infty(S_\Omega)$ is defined by 
 \[T_{\phi}h:= P_{H^2(S_{\Omega})}({\phi}h), \;\; \text{ for every } h \in H^2(S_\Omega),  \]
 where $P_{H^2(S_\Omega)}: L^2(S_\Omega) \to H^2(S_\Omega)$ is the Szeg\"o projection.

For our convenience, throughout the rest of the paper, we will adopt a slight abuse of notation by denoting
$\left\{\sqrt{\prod_{j=1}^{\ell} (1+\frac{a}{2}(j-1))}\psi_{\alpha}^{(\ell)}\right\}$ as $\{\psi_{\alpha}^{(\ell)}\}$ for each $\ell=1,\ldots,r.$ With this convention, the equation \eqref{delk} can be rewritten as  \beq \label{newdelk} \Delta^{(\ell )}(z,w) =  \sum_\alpha \psi_{\alpha}^{(\ell)}(z) \overline{\psi_{\alpha}^{(\ell)}(w)}, \;\; \text{ for all } 1 \leq \ell \leq r.\eeq 
Now $\{\psi_{\alpha}^{(\ell)}\}$ in the equation \eqref{newdelk} is no longer an orthonormal basis, rather an orthogonal basis of $\mathcal P_{(\ell)}$.
Also, we denote $\boldsymbol T_{z}=(T_{z_1}, \ldots, T_{z_d})$ as the tuple of multiplication operators by the coordinate functions $z_1, \ldots, z_d$ on $H^2(S_\Omega)$ and $\boldsymbol M_z= (M_{z_1}, \ldots, M_{z_d})$ as the tuple of operators of multiplication by the co-ordinate functions $z_1,\ldots,z_d$ on $L^2(S_\Omega)$. 

The next two results have been recorded for the Hardy space $H^2 (\mathbb{B}_d)$ in \cite{DJ}. We now extend to a more general context for the Hardy space $H^2(S_\Omega).$ The idea of the proof of the next result is motivated by \cite[Lemma 2.5]{DJ}.
\begin{lemma} \label{bh1}
  Let $X \in \mathcal{B}(H^2(S_\Omega))$ be such that $\sum_\alpha \psi_\alpha^{(\ell)}(\boldsymbol{T}_z)^* X \psi_\alpha^{(\ell)}(\boldsymbol{T}_z)= \binom{r}{\ell}X $ for all $1 \leq \ell \leq r,$ where $r$ is the rank of the domain $\Omega.$  Then there exists $A \in \mathcal{B}(L^2(S_\Omega))$ such that \[\sum_\alpha \psi_\alpha^{(\ell)}(\boldsymbol{M}_z)^* A \psi_\alpha^{(\ell)}(\boldsymbol{M}_z)= \binom{r}{\ell}A, \; \; \; 1 \leq \ell \leq r,\] with $\|A\|=\|X\|$ and $X=P_{H^2(S_\Omega)}A|_{H^2(S_\Omega)}$.
\end{lemma}
\begin{proof}

  Let $\Phi_{\ell}: \mathcal{B}(L^2(S_\Omega)) \to \mathcal{B}(L^2(S_\Omega))$ defined as $$\Phi_\ell(B):= \frac{1}{\binom{r}{\ell}}\sum_{\alpha} \psi_{\alpha}^{(\ell)} (\boldsymbol{M}_z)^*B\psi_{\alpha}^{(\ell)} (\boldsymbol{M}_z), \;\;\; 1 \leq \ell \leq r \text{ and } B \in \mathcal{B}(L^2(S_\Omega)).$$ By Lemma \ref{lemmaSC}, $\sum_{\alpha}\psi_{\alpha}^{(\ell)}(\boldsymbol{M}_z)^*\psi_{\alpha}^{(\ell)}(\boldsymbol{M}_z)=\binom{r}{\ell} I.$ Thus, we get $\|\Phi_{\ell}(B)\| \leq \|B\|.$ 
Choose an operator $\Tilde{X} \in \mathcal{B}(L^2(S_\Omega))$ such that $X=P_{H^2(S_\Omega)}\tilde{X}|_{H^2(S_\Omega)}$ and $\|\tilde{X}\|=\|X\|$.
For $n \in \mathbb N$, define $A_n^{(1)} :=\frac{1}{n}\sum_{t=1}^n \Phi_{1}^t(\tilde{X}).$ Then for any $h, g \in L^2(S_\Omega)$, we have
 \beqn
	|\inp{A_n^{(1)}h}{g}|
			\leq  \frac{1}{n}\sum_{t=1}^{n}\|\Phi_{1}^t (\tilde{X})\| \|h\| \|g\|
			\leq  \|\tilde{X}\| \|h\| \|g\|= \|X\| \|h\| \|g\|.\\		
\eeqn
Consequently, the sequence $\{\inp{A_n^{(1)} h}{g}\}$ has a convergent subsequence. By the separability of $L^2(S_\Omega)$, the sequence $\{A_n^{(1)}\}$ admits a WOT limit point $A^{(1)}$ in $\mathcal{B}(L^2(S_\Omega))$. 
A routine verification shows that $P_{H^2(S_\Omega)} A^{(1)}|_{H^2(S_\Omega)}=X$, $\sum_{\alpha}\psi_{\alpha}^{(1)}(\boldsymbol{M}_z)^* A^{(1)}\psi_{\alpha}^{(1)}(\boldsymbol{M}_z)=\binom{r}{1}A^{(1)}$ and $\|A^{(1)}\|=\|X\|$.

We now define another sequence of operators as follows: $A_n^{(2)}= \frac{1}{n}\sum_{t=1}^n \Phi_{2}^t(A^{(1)})$. Using the same approach as above, let $A^{(2)}$ be a WOT limit point of $\{A_n^{(2)}\}$, then it is easy to see that $P_{H^2(S_\Omega)} A^{(2)}|_{H^2(S_\Omega)}=X$, $\|A^{(2)}\|=\|X\|$, and $$\sum_{\alpha}\psi_{\alpha}^{(\ell)}(\boldsymbol{M}_z)^* A^{(2)}\psi_{\alpha}^{(\ell)}(\boldsymbol{M}_z)=\binom{r}{\ell}A^{(2)}\;\; \text{for } \ell=1,2.$$
If $A^{(r)}\in \mathcal{B}(L^2(S_\Omega))$ is the operator obtained after the $r$th step of the above process,
then the conclusion follows by taking $A=A^{(r)}$. 

\end{proof}
\begin{remark}
    Note that $\Phi_\ell (\Phi_{\ell'}(B))= \Phi_{\ell'}(\Phi_{\ell}(B))$, $1 \leq \ell, \ell'\leq r$ and $B \in \mathcal{B}(L^2(S_\Omega)).$ The process used in the proof of the above lemma can be started at any step $\ell \in \{1,\ldots,r\}$. And then after completing all the $r$ steps, one can get the conclusion of the lemma.
\end{remark}
The following result can be obtained by a suitable scaling in \cite[Proposition 2.4]{DJ}. Hence, we omit the proof.
\begin{proposition}\label{bh2}
    Let $(N_1,\ldots,N_d)$ be a commuting normal tuple on a Hilbert space $\mathcal{K}$ such that $\sum_{i=1}^d N_i^* N_i =rI$. If $A \in \mathcal{B}(\mathcal{K})$ satisfies $\sum_{i=1}^d N_i^* A N_i =r A$, then $A$ commutes with $N_i$ and $N_i^*$ for all $1 \leq i \leq d$.
\end{proposition}
The following result is a several variable analog of \cite[Corollary 12.7]{Conway}. We provide a proof here for the sake of completeness.
\begin{lemma}\label{bh3}
    Let $\mu$ be a finite positive regular Borel measure on a compact set $K \subseteq \mathbb{C}^d$. Let $A \in \mathcal{B}(L^2(\mu))$ commute with each multiplication operator $M_{z_i},\;i=1,\ldots,d$. Then there exists a function $\phi \in L^\infty (\mu)$ such that $A=M_\phi$.  
\end{lemma}
\begin{proof}
 Let $A \in \mathcal{B}(L^2 (\mu))$ commutes with each $M_{z_i}$ then by Fuglede-Putnam theorem,  $A M_{p}= M_{p}A$ for any polynomial $p \in \mathbb C[z, \bar{z}]$. As $K$ is compact, Hausdorff and $\mu$ is regular, by \cite[Corollary 4.53]{Doug}, polynomials in $z$ and $\bar{z}$ are weak* dense in $L^\infty(\mu)$. Therefore,  $A M_\varphi = M_\varphi A$ for all $\varphi \in L^\infty(\mu).$ Since $L^\infty(\mu)$ is a maximal abelian von-Neumann algebra (see  \cite[Proposition 4.50]{Doug}), by \cite[Proposition 4.62]{Doug}, the commutant of $\{M_\varphi: \varphi \in L^\infty(\mu)\}$ coincides with itself. This completes the proof.
\end{proof}
We now give a Brown-Halmos type characterization for the Toeplitz operator on the Hardy space $H^2(S_\Omega).$
\begin{theorem}\label{toeplitz}
Let $\boldsymbol T_{z}=(T_{z_1}, \ldots, T_{z_d})$ be the tuple of multiplication operators by the coordinate functions $z_1, \ldots, z_d$ on $H^2(S_\Omega).$
A necessary and sufficient condition for an operator $X$ on $H^2(S_\Omega)$ to be a Toeplitz operator is that \beq \label{b-h}\sum_\alpha \psi_\alpha^{(\ell)}(\boldsymbol{T}_z)^* X \psi_\alpha^{(\ell)}(\boldsymbol{T}_z)= \binom{r}{\ell}X,\; \; \; 1 \leq \ell \leq r,\eeq
 where $\{\psi_{\alpha}^{(\ell)}\}$ is the same as that appearing in the equation \eqref{newdelk}.
\end{theorem}

\begin{proof}
Let $X=T_{\phi}$ be a Toeplitz operator for some $\phi \in  L^\infty(S_\Omega).$ For any $1 \leq \ell \leq r$ and $ h, g \in H^2(S_\Omega)$, 
\beqn 
\Inp{\sum_{\alpha}\psi_{\alpha}^{(\ell)}(\boldsymbol{T}_z)^*X\psi_{\alpha}^{(\ell)}(\boldsymbol{T}_z)h}{g}_{H^2(S_\Omega)} &=& \sum_{\alpha} \Inp{X \psi_{\alpha}^{(\ell)}(\boldsymbol{T}_z)h}{\psi_{\alpha}^{(\ell)}(\boldsymbol{T}_z)g}_{H^2(S_\Omega)}\\
&=& \sum_{\alpha} \Inp{\phi \psi_{\alpha}^{(\ell)}h}{\psi_{\alpha}^{(\ell)}g}_{L^2(S_\Omega)}\\
&\overset{\eqref{newdelk}}=& \int_{S_\Omega} \Delta^{(\ell)}(z,z) \phi(z)h(z) \overline{g(z)}dz\\
&=& \binom{r}{\ell} \Inp {\phi h}{g}_{L^2(S_\Omega)}= \binom{r}{\ell}\Inp{Xh}{g}_{H^2(S_\Omega)}.\\
\eeqn	
Conversely, suppose that $X \in \mathcal{B}(H^2(S_{\Omega}))$ satisfies the conditions given in equation \eqref{b-h}. Then, by Lemma \ref{bh1}, there exists a bounded operator $A$ on $L^2(S_\Omega)$ such that $X = P_{H^2(S_\Omega)}A|_{H^2(S_\Omega)}$ and \[\sum_\alpha \psi_\alpha^{(\ell)}(\boldsymbol{M}_z)^* A \psi_\alpha^{(\ell)}(\boldsymbol{M}_z)= \binom{r}{\ell}A ,\; \; 1 \leq \ell \leq r.\]
In particular, for $\ell=1$, we have $\sum_{i=1}^d M_{z_i}^* A M_{z_i}=rA$. By Proposition \ref{bh2}, $A$ commutes with each $M_{z_i}$. Thus, it follows from Lemma \ref{bh3} that $A=M_\phi$ for some $\phi \in L^\infty(S_\Omega)$. Now the conclusion is immediate. 

\end{proof}
A careful observation of the proofs of Lemma \ref{bh1} and Theorem \ref{toeplitz} yields the following sufficient condition for Toeplitz operators on $H^2(S_\Omega).$
\begin{corollary} \label{rem1}
If an operator $X\in \mathcal{B}(H^2(S_\Omega))$ satisfies $\sum_{i=1}^d T_{z_i}^*XT_{z_i}=rX$, then $X$ is a Toeplitz operator.
\end{corollary}
Brown and Halmos observed that every compact Toeplitz operator on the Hardy space of the unit disc is the zero operator (see the Corollary on page 94 of \cite{Brown-Halmos}). This conclusion for the polydisc and the unit ball follows from \cite[Theorem 3.2]{MSS} and \cite[Theorem 3.3]{DE}, respectively. In the following result, we show that the same conclusion holds for the classical Cartan domains.
\begin{theorem} \label{ctoperator}
The only compact Toeplitz operator on the Hardy space $H^2(S_\Omega)$ is the zero operator.
\end{theorem}
To prove the above result, we follow the approach of the Berezin transform used in \cite[Section 1]{Englis}.
Before presenting the proof, we first recall a few results related to the generalized Cayley transform.
Let $c$ denote the generalized Cayley transform acting on the Cartan factor $Z$, then $c$ maps the domain $\Omega$ biholomorphically onto a generalized upper half plane $\mathcal{D}$ (see \cite[Theorem 6.8]{K-W}). Let $\Sigma$ be the Shilov boundary of $\mathcal{D}$ and $\beta$ be the associated normalized measure supported on $\Sigma$. By \cite[Theorem 6.9]{K-W}, $c^{-1}(\Sigma)$ is a dense open subset of $S_\Omega$. Therefore,  the set $S_\Omega \setminus c^{-1}(\Sigma)$ is of measure zero with respect to the unique $\mathbb K$-invariant measure $dz$ on $S_{\Omega}.$ Consequently, $f \in L^\infty(S_\Omega)$ if and only if $f \in L^\infty(c^{-1}(\Sigma))$. Moreover, the Cayley transform $c$ extends continuously on $c^{-1}(\Sigma)$ (see the discussion after \cite[Proposition 1.1]{upmeier3}). Let $\xi \in c^{-1}(\Sigma)$. Since the Cayley transform $c$ is continuous on $c^{-1}(\Sigma),$ one can choose a sequence $\{z_n\}$ in $\Omega$ such that $z_n \to \xi$ and  $c(z_n) \to c(\xi) \in \Sigma$ {\it admissibly} and {\it restrictedly}. We refer to \cite[Section 2]{Weiss} for more details on the admissibly and restrictedly notions of convergence.

Note that the Szeg\"o kernel on $\Omega$ is given by $K^{(\frac{d}{r})}(z,w)=\Delta(z,w)^{-\frac{d}{r}}$ for all $z,w \in \Omega.$ 
Let $P: S_\Omega \times \Omega \to \mathbb C$ be the Poisson kernel for the domain $\Omega$ defined as \[P(\xi, z):=\frac{|K^{(\frac{d}{r})}(\xi, z)|^2}{K^{(\frac{d}{r})}(z,z)},\;\; \xi \in S_{\Omega}, z \in \Omega.\]
If $\mathcal P: \Sigma \times \mathcal{D} \to \mathbb C$ is the Poisson kernel associated to the generalized upper half plane $\mathcal{D},$ then 
by \cite[equation 4.3]{Koranyi} (see also \cite[Proposition 4.10]{Koranyi}), we get \beq \label{PoissonRelation} P(\xi, z)= \frac{\mathcal P\big(c(\xi),c(z)\big)}{\mathcal P\big(c(\xi), c(0)\big)}, \; \; \xi \in c^{-1}(\Sigma), z \in \Omega. \eeq
Furthermore, for a fixed $z \in \Omega$,  $P_z(\xi):=P(\xi, z)$ is  continuous on $S_{\Omega},$ and for any fixed $z \in \mathcal{D}$, $ \mathcal P_z(u):= \mathcal P(u,z)$ is continuous on $\Sigma$. 
Let $\phi \in L^\infty(S_\Omega)$, the Poisson transform of $\phi$ is given by \beqn \label{PoissonC}P[\phi](z)= \int_{S_\Omega}\phi(\xi)P(\xi,z) dz(\xi).\eeqn Similarly, the Poisson transform of $\phi \in L^\infty(\Sigma)$ is given by \beqn \label{PoissonH}\mathcal P[\phi](z)= \int_{\Sigma}\phi(u)\mathcal{P}(u,z) d\beta(u).\eeqn The relation between $d\beta$ and $dz$ is given by $dz(\xi)= \mathcal P_{c(0)}(c(\xi)) d\beta(c(\xi)),$ for any $\xi \in c^{-1}(\Sigma)$ (see \cite[equation 4.1]{Koranyi}).

\begin{proof}[Proof of the Theorem \ref{ctoperator}]
    Let $T_\phi$ be a compact Toeplitz operator on $H^2(S_\Omega)$ for some $\phi \in L^\infty(S_\Omega)$. For any polynomial $p \in H^2(S_\Omega)$, $\inp{p}{g_z}= \Delta(z,z)^{\frac{d}{2r}}p(z),$ where $g_z (\cdot)=\frac{K^{(\frac{d}{r})}(\cdot,z)}{\|K^{(\frac{d}{r})}(\cdot,z)\|}.$  By Lemma \ref{Shilov},  $\Delta (z,z)= \sum_{\ell =0}^r (-1)^{\ell} \Delta^{(\ell)}(z,z) \to \sum_{\ell =0}^r (-1)^{\ell} \binom{r}{\ell}= 0$ as $z \to \xi \in S_\Omega.$ Thus $\inp{p}{g_z} \to 0$  as $z \to \xi \in S_\Omega$. Since polynomials are dense in $H^2(S_\Omega)$, $g_z \overset{w}\to 0$ (weakly) as $z \to \xi \in S_\Omega$.  Consequently, $T_\phi(g_z) \to 0$ as $z \to \xi$ and therefore the Berezin transform $\Tilde{T}_\phi$ of $T_{\phi}$ given by $\Tilde{T}_\phi(z): =\inp{T_\phi g_z}{g_z} \to 0$ as $z \to \xi \in S_\Omega$. 
  Note that \[\Tilde{T}_\phi (z) = \int_{S_\Omega}\phi(\xi) |g_z(\xi)|^2dz(\xi)=\int_{S_\Omega}\phi(\xi) P(\xi, z)dz(\xi)=P[\phi](z).\]
This shows that $P[\phi](z) \to 0$ whenever $z \to \xi \in S_\Omega.$ 

Let $A: L^\infty(S_\Omega) \to L^\infty (\Sigma)$ be a linear map defined as \[ (Af)(u):= (\mathcal P_{c(0)}(u))^{1/2} f (c^{-1}(u)), \;\; \text{ for any } u \in \Sigma,\;\; f \in L^\infty(S_\Omega).\] 
Following the proof of \cite[Lemma 4.4]{Koranyi}, we have $\inp{Af}{Ah}_{L^2(\Sigma)}=\inp{f}{h}_{L^2(S_\Omega)}$ for all $f, h \in L^\infty(S_\Omega).$
    
Therefore, for any $\phi \in L^\infty(S_\Omega)$, we get
\beqn
    P[\phi](z) &=& \int_{S_\Omega}  \phi (\xi)P(\xi, z) dz(\xi)= \inp{\phi}{P_z}_{L^2(S_\Omega)}= \inp{A\phi}{A P_z}_{L^2(\Sigma)}\\
    &=& \int_{\Sigma} (\mathcal P_{g(0)}(u))^{1/2} \phi (c^{-1}(u)) (\mathcal P_{c(0)}(u))^{1/2} P_z(c^{-1}(u))d\beta(u)\\
    &\overset{\eqref{PoissonRelation}}=& \int_{\Sigma} \phi(c^{-1}(u)) \mathcal P(u, c(z)) d\beta(u)
    = \mathcal P[\phi \circ c^{-1}](c(z)).\eeqn 
Let $\{z_n\}$ be a sequence in $\Omega$ such that $z_n \to \xi \in c^{-1}(\Sigma)$ and  $c(z_n) \to c(\xi) \in \Sigma$ admissibly and restrictedly. Then by \cite[Theorem 2]{Weiss}, $\mathcal P[\phi \circ c^{-1}] (c(z_n)) \to \phi \circ c^{-1}(c(\xi))=\phi(\xi)$ for almost every $\xi \in c^{-1}(\Sigma)$. Therefore, for almost every $\xi \in S_\Omega$, $\phi(\xi)= \lim_{n \to \infty}P[\phi](z_n)=0$. This completes the proof.
\end{proof}

\subsection{$\boldsymbol T$-Toeplitz and Reflexivity }
Recall that a commuting $d$-tuple $\boldsymbol T$ of operators $T_1,\ldots,T_d$ defined on $\mathcal{H}$ is a Cartan isometry  if and only if \[\sum_{\alpha} \psi_{\alpha}^{(\ell)}(\boldsymbol{T})^* \psi_{\alpha}^{(\ell)}(\boldsymbol{T})=\binom{r}{\ell}I_{\mathcal{H}},\;\; 1 \leq \ell \leq r .\]
For $\ell=1$, we have $\sum_{i=1}^d T_i^*T_i=rI_{\mathcal{H}}$.
  Motivated by Corollary \ref{rem1}, we define the notion of $\boldsymbol{T}$-{\it Toeplitz operator} as follows: 
  \begin{definition}
   An operator $X\in \mathcal{B}(\mathcal{H})$ is said to be $\boldsymbol{T}$-Toeplitz for a Cartan isometry $\boldsymbol T$ if $\sum_{i=1}^d T_i^* X T_i = rX,$ where $r$ is the rank of the domain $\Omega.$
  \end{definition}
  From now on, we denote $\mathcal{T}(\boldsymbol{T})$ as the set of all $\boldsymbol{T}$ -Toeplitz operators associated to a Cartan isometry $\boldsymbol T.$ 

Let $\boldsymbol{S}=(S_1, \ldots, S_d)$ be a $d$-tuple of commuting operators defined on $\mathcal H$ that satisfies $\sum_{i=1}^d S_i^*S_i=rI_{\mathcal{H}}$. By setting $\phi:\mathcal{B}(\mathcal{H}) \to \mathcal{B}(\mathcal{H})$ as  $\phi(A)= \frac{1}{r}\sum_{i=1}^d S_i^* AS_i$ in the proof of part (3) of the Theorem 1.2 in \cite{BP}, the next result is immediate. 
\begin{lemma}\label{prunaru}
    Let $\boldsymbol{S}=(S_1, \ldots, S_d)$ be a commuting $d$-tuple which satisfies the relation $\sum_{i=1}^d S_i^*S_i=rI_{\mathcal{H}}$ and let $\boldsymbol{U}=(U_1,\cdots,U_d)$ be its minimal normal extension defined on $\mathcal{K}$ containing $\mathcal{H}.$ Then for any $X \in \mathcal{B}(\mathcal{H})$ satisfying $\sum_{i=1}^d S_i^* X S_i= rX$
    must be of the form $P_{\mathcal{H}}Y|_{\mathcal{H}}$ for some $Y \in (\boldsymbol{U})'$, the commutant of $\boldsymbol U.$    
\end{lemma}
As in the case of Toeplitz operators on $H^2(S_{\Omega})$ (see Theorem \ref{toeplitz}), below we obtain a similar result for $\mathcal{T}(\boldsymbol{T})$.
\begin{proposition}
   Let $\boldsymbol{T}$ be a Cartan isometry and $X \in \mathcal{T}(\boldsymbol{T})$ then for any $1 \leq \ell \leq r$, \[\sum_{\alpha} \psi_{\alpha}^{(\ell)}(\boldsymbol{T})^* X\psi_{\alpha}^{(\ell)}(\boldsymbol{T})=\binom{r}{\ell}X.\]
\end{proposition}
\begin{proof}
    Let $\boldsymbol{N}$ be the minimal normal extension of $\boldsymbol{T}$ defined on $\mathcal{K}.$ Note that $\boldsymbol T$ satisfies $\sum_{i=1}^d T_i^*T_i=rI_{\mathcal{H}}.$ Let $X \in \mathcal{T}(\boldsymbol{T}),$ then by Lemma \ref{prunaru}, there exists an operator $Y$ in the commutant $(\boldsymbol{N})'$ such that $X$ is the compression of $Y$ to $\mathcal H.$
For any $h,g \in \mathcal{H}$ and $1 \leq \ell \leq r$,  Lemma \ref{lemmaSC} combined with Theorem \ref{characterization} gives
\beqn
 \Inp{\sum_{\alpha} \psi_{\alpha}^{(\ell)}(\boldsymbol{T})^* X\psi_{\alpha}^{(\ell)}(\boldsymbol{T})h}{g}_{\mathcal{H}}
            &=& \sum_{\alpha}\Inp{ X\psi_{\alpha}^{(\ell)}(\boldsymbol{N})h}{\psi_{\alpha}^{(\ell)}(\boldsymbol{N})g}_{\mathcal{H}}\\
            &=& \sum_{\alpha}\Inp{ Y\psi_{\alpha}^{(\ell)}(\boldsymbol{N})h}{\psi_{\alpha}^{(\ell)}(\boldsymbol{N})g}_{\mathcal{K}}\\
            &=&\sum_{\alpha}\Inp{\psi_{\alpha}^{(\ell)}(\boldsymbol{N})^* Y\psi_{\alpha}^{(\ell)}(\boldsymbol{N})h}{g}_{\mathcal{K}}\\
            &=& \binom{r}{\ell}\inp{Yh}{g}_{\mathcal{K}}
            =\binom{r}{\ell}\inp{Xh}{g}_{\mathcal{H}}.\eeqn
This completes the proof.
\end{proof}

\begin{remark}\label{bhcommutant}
   For a Cartan isometry $\boldsymbol{T}$, the set of all $\boldsymbol T$-Toeplitz operator, $\mathcal{T}(\boldsymbol{T})$ is nothing but  \beqn \left \{ X \in \mathcal{B}(\mathcal{H}) : \sum_{\alpha} \psi_{\alpha}^{(\ell)}(\boldsymbol{T})^* X\psi_{\alpha}^{(\ell)}(\boldsymbol{T})=\binom{r}{\ell}X, 1 \leq \ell \leq r\right\}= \left\{P_{\mathcal{H}} Y|_{\mathcal{H}}: Y \in (\boldsymbol{N})'\right \}. \eeqn
\end{remark}

Note that $\phi(\boldsymbol T)$ is a Cartan isometry whenever $\boldsymbol T$ is a Cartan isometry (see Theorem \ref{invariance}). First we observe that $\phi(\boldsymbol N)$ is the minimal normal extension of $\phi(\boldsymbol T),$ where the $d$-tuple $\boldsymbol N$ defined on $\mathcal K\supseteq \mathcal{H}$ is the minimal normal extension of $\boldsymbol T.$  Clearly, $\phi(\boldsymbol N)$ is a normal extension of $\phi(\boldsymbol T).$ If there exists a reducing subspace $\mathcal{K}' \subseteq \mathcal{K}$ for $\phi(\boldsymbol N)$ containing $\mathcal{H}$ then $\phi(\boldsymbol{N})|_{\mathcal{K}'}$ is a normal extension of $\phi(\boldsymbol{T}).$ Subsequently, $\phi^{-1}(\phi(\boldsymbol{N})|_{\mathcal{K}'})$ becomes a normal extension of $\boldsymbol T.$ This contradicts the minimality of $\boldsymbol N.$
In view of the preceding discussion, we prove the following result. 
\begin{corollary}
 Let $\boldsymbol T$ be a Cartan isometry and $\phi \in \rm{Aut}(\Omega)$, then $\mathcal{T}(\boldsymbol T)= \mathcal{T}(\phi(\boldsymbol T))$.
\end{corollary}
\begin{proof}
By the Remark \ref{bhcommutant}, it suffices to show that $(\boldsymbol{N})'= (\phi(\boldsymbol{N}))'$. Let $Y \in (\boldsymbol N)'$, then  $Y p(\boldsymbol N)=p(\boldsymbol N) Y,$ for any polynomial $p \in \mathbb C[z].$ Since $\phi$ is analytic in a neighbourhood of $\bar{\Omega},$ thus can be approximated uniformly by analytic polynomials on $\sigma(\boldsymbol N).$ Now the conclusion follows from the functional calculus for $\boldsymbol N.$
\end{proof}

For an irreducible bounded symmetric domain $\Omega$, let $A(\Omega)$ denote the algebra of functions holomorphic in $\Omega$ and continuous up to the boundary. Note that $A(\Omega)$ is a closed subalgebra containing polynomials with the Shilov boundary $S_{\Omega}$.  Let $\boldsymbol{N}$ be the minimal normal extension of a Cartan isometry $\boldsymbol{T}$ and let $\mu$ be the associated scalar spectral measure. Let $ \Psi_{\boldsymbol N}$ be the $L^\infty$-functional calculus. The restriction algebra \[\mathcal{R}(\boldsymbol{T}):= \{ f \in L^\infty (\mu) : \Psi_{\boldsymbol N}(f)\mathcal{H} \subseteq \mathcal{H}\}\] is a weak* closed subalgebra of $L^\infty(\mu).$
 Let $A=A(\Omega)|_{S_{\Omega}}.$ Since $\Omega$ is polynomially convex, by an application of the Oka-Weil theorem, the analytic polynomials are dense in $A \subseteq C(S_{\Omega}).$ Consequently, we have $A \subseteq \mathcal{R}(\boldsymbol{T}).$  
 Thus a Cartan isometry $\boldsymbol T$ can be seen as an $A$-isometry (see \cite[Definition 2.1 ]{DEE}).  
Now a Cartan isometry is realized as $A$-isometry, one may define the $\boldsymbol T$-Toeplitz operator described in \cite[Definition 3.1]{DEE}.
We illustrate in Proposition \ref{Jequiv} that our notion of $\boldsymbol T$-Toeplitz operators is consistent with it.

Recall that the dual subalgebra $H^\infty (\mu)$ is the weak* closure of $\{[p|_{S_\Omega}]_\mu : p \in \mathbb{C}[z]\}$ in $ L^\infty (\mu),$ where $[\cdot]_\mu$ is equivalence class in 
 $L^\infty(\mu)$ and $\mathbb C[z]$ denote the algebra of polynomials in $z_1, \ldots, z_d.$ The unital dual operator algebra $\mathcal{A}_{\boldsymbol{T}}$ generated by $\boldsymbol{T},$ is the weak* closure of $\{p(\boldsymbol{T}): p \in \mathbb{C}[z]\}$ in $\mathcal{B}(\mathcal{H}).$
 It is known that $\Psi_{\boldsymbol{N}}$ induces a dual algebra isomorphism, that is, an isometric isomorphism and weak* homeomorphism, \[\gamma_{\boldsymbol{T}}: H^\infty(\mu) \to \mathcal{A}_{\boldsymbol{T}}, \text{ given by } f\to \Psi_{\boldsymbol{N}}(f)|_{\mathcal{H}},\] a unique extension of the polynomial functional calculus of $\boldsymbol{T}$.
 A function $\theta \in H^\infty(\mu)$ is called {\it $\mu$-inner function} if $|\theta|=1 ~\mu$- a.e. on $S_{\Omega}$. Let $I_{\mu}$ be the set of all $\mu$-inner functions in $H^\infty(\mu)$. 
\begin{proposition}\label{Jequiv}
	Let $\boldsymbol{T}$ be a Cartan isometry. An operator $X \in \mathcal{T}(\boldsymbol{T})$ if and only if \[J^*XJ=X \quad \text{for every isometry } J \text{ in the dual algebra } \mathcal{A}_{\boldsymbol{T}}.\]
\end{proposition}
 \begin{proof}
 Using an argument similar to the proof of \cite[Lemma 1.1]{DE}, it can be shown that an operator $J \in \mathcal{A}_{\boldsymbol{T}}$ is an isometry if and only if $J=\gamma_{\boldsymbol{T}}(\theta)$ for some $\mu$-inner function $\theta \in H^\infty (\mu)$. 
 Also, note that by \cite[Proposition 18]{Al} and the subsequent discussion, $(A({\Omega})|_{S_\Omega}, S_{\Omega}, \mu)$ is regular in the sense of Aleksandrov. 
 Thus, by \cite[Corollary 2.5]{DE}, $L^\infty(\mu)= \overline{LH}^{w*}\{\bar{\eta}\cdot\theta: \eta, \theta \in I_{\mu}\}$, where $\overline{LH}^{w*}$ denotes the weak* closed linear span. Now, the rest of the proof can be completed by going along the same lines as the proof of \cite[Proposition 3.1]{DE}.
 \end{proof}
 
\begin{remark}
For a Cartan isometry $\boldsymbol{T},$ consider the set $\mathcal{T}_a(\boldsymbol{T})= \{ P_{\mathcal{H}} \Psi_{\boldsymbol{N}}(f)|_{\mathcal{H}}: f \in H^\infty (\mu)\}$ of all analytic $\boldsymbol T$-Toeplitz operators. For $1\leq p < \infty,$ let $\mathcal{S}_p$ denote the {\it Schatten $p$ class} operators.  Let $S \in \mathcal B(\mathcal H).$ 
\begin{enumerate}
  \item [(i)] If $SX-XS \in \mathcal{S}_p$ for all $X \in\mathcal{T}_a(\boldsymbol{T})$ then $S=Y+K$ with $Y \in \mathcal T(\boldsymbol T)$ and  $K \in \mathcal{S}_p$ (see \cite[Theorem 3]{DES}).   
  \item [(ii)] If $SX-XS$ is a finite rank operator for all $X \in\mathcal{T}_a(\boldsymbol{T})$ then $S$ is a finite rank perturbation of a $\boldsymbol T$-Toeplitz operator provided $\sigma_p(\boldsymbol T)$ is empty (refer to \cite[Theorem 8]{DES}).
\end{enumerate}
\end{remark}

 Let $\boldsymbol{T}$ be a Cartan isometry, and let $\mu$ be the associated scalar spectral measure. Consider the decomposition $\mu=\mu_c+ \mu_d$ of $\mu$ into its continuous and discrete part. This will lead to an orthogonal decomposition $\mathcal{H}=\mathcal{H}_c \oplus \mathcal{H}_d,$  where these subspaces are reducing for every $\boldsymbol{T}$-Toeplitz operator. 
 Note that $\sigma_p(\boldsymbol{T})=\sigma_p(\boldsymbol{N}),$ where $\boldsymbol N$ is the minimal normal extension of $\boldsymbol T.$ For each $\lambda \in \sigma_p(\boldsymbol{T}),$ consider the eigenspace $\mathcal{H}_d^\lambda= \bigcap_{i=1}^d \ker (T_i-\lambda_i)= \bigcap_{i=1}^d \ker (N_i-\lambda_i)$ and $\mathcal{H}_d= \bigoplus_{\lambda \in \sigma_p(\boldsymbol{T})}\mathcal{H}_d^\lambda$. Clearly, $\mathcal{H}_d$ is reducing for $\boldsymbol T.$  Let $\mathcal{H}_c$ denote the orthogonal complement of $\mathcal{H}_d$ in $ \mathcal{H}$. 
We now give an analog of \cite[Proposition 3.2]{DE} for Cartan isometry.
\begin{proposition}
    Let $\boldsymbol{T}$ be a Cartan isometry. Then the orthogonal decomposition $\mathcal{H}=\mathcal{H}_c \oplus \mathcal{H}_d$  induces a decomposition of all $\boldsymbol T$-Toeplitz operators; \[\mathcal{T}(\boldsymbol{T}) \cong \mathcal{T}(\boldsymbol{T}_c) \oplus \bigoplus_{\lambda \in \sigma_p(\boldsymbol{T})} \mathcal{B}(\mathcal{H}_d^\lambda), \] 
    where $\boldsymbol{T}_c= \boldsymbol{T}|_{\mathcal{H}_c}$ is the continuous part of $\boldsymbol{T}$.
\end{proposition}
\begin{proof} We invoke the idea of the proof of \cite[Proposition 3.2]{DE}.
For $\lambda \in \sigma_p(\boldsymbol{T})\subseteq S_{\Omega}$,  let $P_\lambda =\Psi_{\boldsymbol N}(\chi_{\{\lambda\}})$ be the orthogonal projection onto the eigenspace $\mathcal{H}_d^\lambda.$ Note that $P_{\lambda}$ commutes with orthogonal projection $P_{\mathcal{H}}$ onto $\mathcal H$. Thus, $P_\lambda|_{\mathcal{H}}$ is an orthogonal projection onto $\mathcal{H}_d^\lambda$. If $X \in \mathcal{T}(\boldsymbol{T})$, then by Remark \ref{bhcommutant}, there exists $A\in (\boldsymbol{N})'$ such that $X=P_{\mathcal{H}}A|_{\mathcal{H}}.$ By spectral theorem,  $AP_\lambda=P_\lambda A$. Consequently, $XP_\lambda|_{\mathcal H}= P_\lambda|_{\mathcal H} X.$ This shows that $X$ reduces $\mathcal{H}_d^\lambda$ for every $\lambda \in \sigma_p(\boldsymbol{T})$. Let $X_c= X|_{\mathcal{H}_c}$, then it is easy to verify that $\sum_{i=1}^d (\boldsymbol{T}_c)_i^* X_c (\boldsymbol{T}_c)_i= r X_c.$

Conversely, assume that $X= X_c \oplus \bigoplus_{\lambda \in \sigma_p(\boldsymbol{T})}X_\lambda \in \mathcal{B}(\mathcal{H}),$ where $X_c \in \mathcal{T}(\boldsymbol{T}_c)$ and $X_\lambda \in \mathcal{B}(\mathcal{H}_d^\lambda)$. For any $h = h_c\oplus \bigoplus_{\lambda \in \sigma_{p}(\boldsymbol{T})}h_\lambda$ in $\mathcal H_c \oplus \mathcal H_d$, we have 
    \beqn
        \sum_{i=1}^d \Inp{T_i^* X T_i h}{h} &=& \sum_{i=1}^d \big \langle (\boldsymbol{T}_c)_i^* X_c (\boldsymbol{T}_c)_i h_c,h_c \big \rangle + \sum_{\lambda \in \sigma_p(\boldsymbol{T})} \sum_{i=1}^d \Inp{(\boldsymbol{T}|_{\mathcal{H}_d^\lambda})_i^* X_\lambda (\boldsymbol{T}|_{\mathcal{H}_d^\lambda})_i h_\lambda}{h_\lambda}\\
        &=& r \inp{X_c h_c}{h_c} + \sum_{\lambda \in \sigma_p(\boldsymbol{T})} \sum_{i=1}^d \inp{ \lambda_i X_\lambda h_\lambda}{\lambda_i h_\lambda}\\
        &=& r \inp{X_c h_c}{h_c} + \sum_{\lambda \in \sigma_p(\boldsymbol{T})} r \inp{ X_\lambda h_\lambda}{ h_\lambda} = r \inp{ X h}{h}.\\
    \eeqn
 This completes the proof.
\end{proof}
The following result can be obtained as a consequence of the above proposition and following the approach used in the proof of \cite[Theorem 3.3]{DE}. Note that the forward part is also mentioned in \cite[Corollary 3.6(b)]{DEE} for an $A$-isometry. 
\begin{corollary}
A Cartan isometry $\boldsymbol{T}$ has empty point spectrum if and only if the only compact $\boldsymbol{T}$-Toeplitz operator is the zero operator.
\end{corollary}
\begin{remark}
Since the Szeg\"o shift has an empty point spectrum, Theorem \ref{ctoperator} can also be deduced from the above corollary.    
\end{remark}

Let $\mathcal{S} \subseteq \mathcal{B}(\mathcal{H})$ be a family of bounded linear operators. We denote $Lat(\mathcal{S})$ to be the lattice of all closed subspaces of $\mathcal{H}$ invariant under $\mathcal{S}$. 
  The set \[Alg Lat(\mathcal{S})= \{A  \in \mathcal{B}(\mathcal{H}): Lat (\mathcal{S})\subseteq Lat(A) \}\] is a WOT closed subalgebra of $\mathcal{B}(\mathcal{H})$ and contains the closed WOT algebra $\mathcal{W}_{\mathcal{S}}$ generated by $\mathcal{S}$ and identity $I_{\mathcal{H}}$.
The family $\mathcal{S}$ is called {\it reflexive} if $Alg Lat(\mathcal{S})=\mathcal{W}_{\mathcal{S}}$. A commuting reflexive family $\mathcal{S} \subseteq \mathcal{B}(\mathcal{H})$ possesses a non-trivial joint invariant subspace. Bercovici in \cite[Theorem 2.4]{Ber} proved that every commuting family $\mathcal{S} \subseteq \mathcal{B}(\mathcal{H})$ of isometries is reflexive. The reflexivity of spherical isometry has been proved in \cite[Theorem 5]{Did}. Since a Cartan isometry is a regular $A$- isometry, the reflexivity of a Cartan isometry follows from \cite[Theorem 1]{Esch1}.
  \begin{proposition}
  Every Cartan isometry is reflexive.
\end{proposition}
\begin{remark}
The above result can also be proved using a similar idea of \cite[Section $3$(A)]{DE}.  
\end{remark}

\subsection{Dual Toeplitz operators}
Let $H^2(S_{\Omega})$ be the Hardy space over an irreducible bounded symmetric domain $\Omega.$
Consider the orthogonal complement $H^2(S_\Omega)^\perp$ of $H^2(S_{\Omega})$ in $L^2(S_\Omega).$ For $\phi \in L^\infty(S_\Omega)$, the dual Toeplitz operator $S_\phi $ on $H^2(S_\Omega)^\perp$ is defined by \[ S_\phi (h):= (I-P_{H^2(S_\Omega)}) \phi h, \quad \text{ for all } h \in H^2(S_\Omega)^\perp. \] 
The Hankel operator $H_\phi: H^2(S_\Omega) \to H^2(S_\Omega)^\perp$ is defined by $H_\phi (f)= (I-P_{H^2(S_\Omega)})(\phi f)$ for all $f \in H^2(S_\Omega).$ 
Therefore, for any $\phi \in L^\infty(S_\Omega)$, the multiplication operator $M_{\phi}$ on $L^2(S_\Omega)=H^2(S_\Omega)\oplus H^2(S_\Omega)^\perp$ has the following decomposition
 \begin{equation}\label{eqdualbh}
      M_\phi=
\begin{pmatrix}
   T_\phi & H_{\bar{\phi}}^*\\
    H_\phi & S_\phi\\
\end{pmatrix}.
 \end{equation}
 
Note that if $\phi \in L^\infty(S_\Omega)$ is holomorphic symbol then $H_\phi=0.$ Thus for any coordinate function $z_i, i=1,\ldots, d,$ we have
\begin{equation}\label{eqdual}
  M_{z_i}=
\begin{pmatrix}
T_{z_i} & H_{\bar{z_i}}^*\\
0 & S_{z_i}\\
\end{pmatrix} \; \text{and} \;
M_{\bar{z_i}}= 
\begin{pmatrix}
    T_{\bar{z_i}} &0\\
    H_{\bar{z_i}}& S_{\bar{z_i}}\\
\end{pmatrix}.
\end{equation}
Recall that a subnormal operator tuple $\boldsymbol{T}$ defined on $\mathcal H$ is called pure if there is no non-zero reducing subspace $M \subseteq \mathcal{H}$ for $\boldsymbol{T}$ such that the restriction operator tuple  $\boldsymbol{T}|_M=(T_1|_M, \ldots, T_d|_M)$ is commuting normal tuple.
\begin{lemma}  
    Consider the commuting tuple $\boldsymbol S_{\bar{z}}=(S_{\bar{z_1}},\ldots,S_{\bar{z_d}})$ of dual Toeplitz operators $S_{\bar{z_1}} \ldots, S_{\bar{z_d}}.$ Then $\boldsymbol S_{\bar{z}}$ is a Cartan isometry with $\boldsymbol M_{\bar{z}}=(M_{\bar{z_i}}, \ldots, M_{\bar{z_d}})$  being its minimal normal extension.
\end{lemma}
\begin{proof}
 It follows from the equation \eqref{eqdual} that $\boldsymbol S_{\bar{z}}$ is a restriction of $\boldsymbol M_{\bar z}$. This shows that $\boldsymbol S_{\bar{z}}$ is a Cartan isometry.  Since the Szeg\"{o} shift $\boldsymbol T_z$ is pure, by \cite[Remark 3]{At3}, $\boldsymbol M_{\bar z}$ is the minimal normal extension of $\boldsymbol S_{\bar{z}}$.
\end{proof}
Now we will give a Brown-Halmos type condition for dual Toeplitz operators.
\begin{theorem}
    A bounded linear operator $X$ on $H^2(S_\Omega)^\perp$ is a dual Toeplitz operator if and only if it satisfies the relation \[ \sum_{i=1}^d S_{\bar{z_i}}^* X S_{\bar{z_i}}=rX.\]
\end{theorem}
\begin{proof}
Let $X$ be the dual Toeplitz operator with symbol $\phi \in L^\infty (S_\Omega)$. Then for any $h, g \in H^2(S_\Omega)^\perp$,
\beqn
\left \langle \sum_{i=1}^d S_{\bar{z_i}}^* X S_{\bar{z_i}}h, g\right \rangle_{H^2(S_\Omega)^\perp} 
&=& \sum_{i=1}^d \inp{X \bar{z_i}h}{\bar{z_i} g}_{H^2(S_\Omega)^\perp}\\
&=& \sum_{i=1}^d \inp{\phi \bar{z_i}h}{\bar{z_i} g}_{L^2(S_\Omega)}\\
&=&  r \int_{S_\Omega} \phi(z) h(z) \overline{g(z)}d\mu(z)
= r \inp{Xh}{g}_{H^2 (S_\Omega)^\perp}.
\eeqn
Conversely, let $X \in \mathcal{B}(H^2(S_\Omega)^\perp)$ satisfy $\sum_{i=1}^d S_{\bar{z_i}}^* X S_{\bar{z_i}}=rX$. Then by remark \ref{bhcommutant}, $X= P_{H^2(S_\Omega)^\perp} Y|_{H^2 (S_\Omega)^\perp}$ for some $Y \in (\boldsymbol M_{\bar z})'= (\boldsymbol M_z)'$. By Lemma \ref{bh3}, $Y= M_\phi$ for some $\phi \in L^\infty (S_\Omega)$. This completes the proof.
\end{proof}

\appendix
\section{Description of other boundary components}

For a tuple $t=(t_1,\ldots,t_r)$ of real variables $t_1,\ldots,t_r$, the $\ell^{th}$-symmetric polynomial $\sigma_{\ell}(t)$ is given by
$\sigma_{\ell}(t)=\sigma_{\ell}(t_1,\ldots,t_r)=\sum_{1 \leq j_1 < \ldots <j_{\ell} \leq r}t_{j_1}\cdots t_{j_\ell}, \; \ell=1,\ldots,r$. For $\ell=0$, we fix the convention that $\sigma_0(t)=1$ even if the set of variables is empty. Before we describe the other boundary components, first, we prove the following lemma using the generating function for elementary symmetric polynomials.  
\begin{lemma} \label{numericalequation}
  Let $q\in \{1,\ldots,r\}$ be a fixed positive integer, then
   \beq \label{sym-poly} \sigma_{\ell}(t) =\sum_{i\geq \rm{max}(\ell-(r-q),0)}^{\rm{min}(q,\ell)} \binom{q}{i}\sigma_{\ell-i}(t_{q+1},\ldots,t_r), \;\; \text{for all } 1 \leq \ell \leq r \eeq  if and only if $t_1=t_2=\cdots=t_q=1$.
\end{lemma}
\begin{proof}
Note that $$ \prod_{i=1}^r (1+t_i x)=\sum_{\ell=0}^r \sigma_{\ell}(t_1,\ldots,t_r) x^\ell.$$ If $t_1=\cdots=t_q=1$ for some positive integer $1 \leq q \leq r,$ then  $$(1+x)^q \left ( \prod_{i={q+1}}^r(1+t_i x)\right )= \sum_{\ell=0}^r \sigma_{\ell}(t_1,\ldots,t_r) x^\ell.$$ 
Equivalently, we get
$$\left(\sum_{i=0}^q \binom{q}{i} x^i\right)\left( \sum_{j=0}^{r-q} \sigma_{j}(t_{q+1},\ldots,t_r)x^j\right)=\sum_{\ell=0}^r \sigma_{\ell}(t_1,\ldots,t_r) x^\ell.$$
By comparing the coefficient of $x^\ell$ on both sides, we get 
$$\sigma_{\ell}(t) =\sum_{i\geq \rm{max}(\ell-(r-q),0)}^{\rm{min}(q,\ell)} \binom{q}{i}\sigma_{\ell-i}(t_{q+1},\ldots,t_r), \;\; \text{for all } 1 \leq \ell \leq r.$$

Conversely, let us assume that the above equation holds.
 By multiplying $x^\ell$ on both sides of the equation and summing over $\ell=0,\ldots,r$, we get 
 \beqn \prod_{i=1}^r (1+t_ix) &=& \sum_{\ell=0}^r \left(\sum_{i\geq \rm{max}(\ell-(r-q),0)}^{\rm{min}(q,\ell)} \binom{q}{i}\sigma_{\ell-i}(t_{q+1},\ldots,t_r)\right)x^\ell\\
 &=& (1+x)^q \left( \prod_{i={q+1}}^r (1+t_i x)\right ).\\
\eeqn
Therefore, $$ (1+x)^q= \prod_{i=1}^q (1+t_i x)= \sum_{\ell=0}^q \sigma_{\ell} (t_1,\ldots,t_q) x^\ell.$$
Now by comparing coefficients of $x^\ell$, we get $\sigma_{\ell}(t_1,\ldots,t_q)=\binom{q}{\ell}$ for all $1 \leq \ell \leq q.$  This shows that $t_1=\cdots=t_q=1.$
\end{proof}

Consider the type-I domain $\Omega=\{z \in \mathbb C^{r \times (r+b)}:  I_r-zz^* >0\}$ of rank $r$, where $I_r$ is the identity matrix of order $r.$ For each $i=1,\ldots,r$, let $e_i$ denote the matrix of order $r \times (r+b)$ with only nonzero entry $1$ at $(i,i)^{th}$ place. The set $\{e_1,\ldots,e_r\}$ forms a Jordan frame (maximal family of orthogonal, minimal tripotents) for $\mathbb C^{r\times(r+b)}.$ 

For a fixed integer $1 \leq q < r$, $w=e_1+\cdots+e_q$ is a tripotent of rank $q$. Let $S_q$ be the set of all tripotents of rank $q$. Since $S_q$ is an irreducible $\mathbb K$-orbit (cf. \cite[page 18]{Arazy}), $S_q=\{k(w): k \in \mathbb K\}$. The $q^{th}$ boundary component $\partial\Omega_{q}$ is given by $\partial\Omega_{q}=\bigcup_{c\in S_q}c+\Omega_c$, where $\Omega_c$ is the intersection of $\Omega$ and the subspace $Z_0(c)$ corresponding to the eigenvalue $0$ in the Peirce decomposition (see \cite[equation 1.5.50]{upmeier}). Therefore, by \cite[equation 2.16]{Arazy}, $\partial\Omega_q= \{k(w+\Omega_w ): k \in \mathbb K\}.$ It follows from the example (1.5.51) in \cite{upmeier} that \[w+\Omega_w=\left \{\begin{pmatrix}
    I_{q} & 0_{q \times (r-q+b)}\\
    0_{(r-q) \times q} & z'
\end{pmatrix}: z' \in \Omega' \right \},\] where $\Omega'=\{z' \in \mathbb C^{(r-q)\times (r-q+b)}: I_{r-q}-z' {z'}^* >0 \}$ is the type-I domain of rank $r-q.$ Let $\mathbb K'$ be the subgroup of linear automorphisms in $\rm{Aut}(\Omega'),$ then (cf. \cite[page 10]{Arazy}) $$\mathbb K'=\left \{ k'=\begin{pmatrix}
    u' & 0\\
    0 & v'
\end{pmatrix}: \rm{det}(k')=1\;\; \text{ where } u' \in \mathcal{U}(r-q),\;\; v' \in \mathcal{U}(r-q+b) \right \}.$$ 
For all $j=1,\ldots,(r-q)$, let $e'_{(q+j)}$ denote $(r-q)\times(r-q+b)$ matrix with only nonzero entry $1$ at $(j,j)^{th}$ place. Then $\{e_{q+1}',\ldots,e_{r}'\}$ forms a Jordan frame for $\Omega'.$ Furthermore,  
$\Omega'=\{k'(\sum_{i=q+1}^r t_i e_i'): k' \in \mathbb K', \;\; 0 \leq t_r \leq \ldots \leq t_{q+1} <1\}$ and 
\[e_{(q+j)}=
\begin{pmatrix}
    0_{q\times q} & 0_{q \times (r-q+b)}\\
    0_{(r-q)\times q} & e'_{(q+j)}
\end{pmatrix}
.\] Consequently, $w+\Omega_w =\{k(e_1+\ldots+e_q+\sum_{i=q+1}^r t_i e_i): 0 \leq t_r \leq \ldots \leq t_{q+1} <1, k \in \mathcal{I}\},$ where $$\mathcal{I}=\left\{k=\begin{pmatrix}
    \begin{pmatrix}
        I_q & 0\\
        0 & u'
    \end{pmatrix} & 0\\
    0 & \begin{pmatrix}
        I_q & 0\\
        0 & v'
    \end{pmatrix}
\end{pmatrix} \in \mathbb K: \begin{pmatrix}
    u' & 0\\
    0 & v'
\end{pmatrix} \in \mathbb K'\right \}.$$ 

Therefore, for a fixed integer $1 \leq q < r$, the $q^{th}$ boundary component $\partial\Omega_q$ of $\Omega$ is given by
\[\partial\Omega_q = \{k(e_1+\cdots+e_q+ \sum_{i=q+1}^r t_i e_i): k \in \mathbb K, 0 \leq t_r \leq \ldots \leq t_{q+1} < 1\}.\] 
The boundary component $\partial\Omega_q$ need not be closed. In fact, the closure of $\partial\Omega_q$ is given as
\beq \label{otherboundary} \overline{\partial\Omega_q} &=& \{k(e_1+\cdots+e_q+ \sum_{i=q+1}^r t_i e_i): k \in \mathbb K, 0 \leq t_r \leq \ldots \leq t_{q+1} \leq 1\}\\
&=& \partial \Omega_{q}\cup \ldots \cup \partial \Omega_r. \nonumber \eeq

Consider a map $\phi: \mathbb C^{r \times (r+b)} \to \mathbb C^{(r-q)\times (r-q+b)}$ defined by \[
\phi \left( \begin{pmatrix}
    u_{11} & u_{12} & v_1\\
    u_{21} & u_{22} & v_2
\end{pmatrix}_{r\times (r+b)}\right)
= \begin{pmatrix}
    u_{22} & v_2
\end{pmatrix}_{(r-q) \times (r-q+b)},  \]
where $u_{11} \in \mathbb C^{q \times q},  u_{12} \in \mathbb C^{q \times (r-q)}, u_{21} \in \mathbb C^{(r-q) \times q },  u_{22} \in \mathbb C^{(r-q) \times (r-q)},  v_1 \in \mathbb C^{q \times b}$ and $v_2 \in \mathbb C^{(r-q) \times b}.$

Each coordinate function of $\phi$ is a linear polynomial and $\phi(\sum_{i=1}^r t_i e_i)= \sum_{i=q+1}^r t_i e'_i$.
We now define a map $\Psi$ as follows:
For each $z=k(\sum_{i=1}^r t_i e_i) \in C^{r \times (r+b)}, $ $$ \Psi(z) :=(\phi \circ k^{-1})(z)= \sum_{i=q+1}^r t_i e'_i.$$ It is evident that $\Psi$ is well defined and coordinate functions of $\Psi$ are linear polynomials. 
In accordance with the Lemma \ref{Shilov}, which describes the Shilov boundary, below we provide a similar description for other boundary components of type-I domain.
\begin{theorem} \label{other-boundary}
    Let $\Omega$ be a type-I domain of rank $r$. Then for any $1 \leq q < r$, \[\overline{\partial\Omega_q}= \left \{z \in \mathbb C^{r \times (r+b)}: \Delta_{\Omega}^{(\ell)}(z,z)= \sum_{i \geq \rm{max}(\ell-(r-q),0)}^{\rm{min}(q,\ell)} \binom{q}{i}\Delta_{\Omega'}^{(\ell-i)}(\Psi(z),\Psi(z)),\;\; 1 \leq \ell \leq r\right \},\] where $\Psi$ is defined as above, $\Omega'$ is a type-I domain of rank $r-q$  and $(\ell-i) \in \vec{\mathbb N}^{(r-q)}$ is the signature of length $(r-q)$ with only first $\ell-i$ many ones.
\end{theorem}
\begin{proof} 
 For any $z=k(\sum_{i=1}^r t_i e_i) \in \mathbb C^{r\times (r+b)}$ and $1 \leq \ell \leq r$, by \cite[Lemma 3.2]{FK}, we have \beqn \Delta_\Omega^{(\ell)}(z,z)&=& \Delta_\Omega^{(\ell)} \left(k(\sum_{i=1}^r t_i e_i),k(\sum_{i=1}^r t_i e_i)\right)= \Delta_\Omega^{(\ell)}(\sum_{i=1}^r t_i e_i, \sum_{i=1}^r t_i e_i)\\
&=&\Delta_{\Omega}^{(\ell)}(\sum_{i=1}^r t_i^2 e_i, e)= \sigma_{\ell}(t^2)\eeqn
Let $z \in \overline{\partial\Omega_q},$ then it follows from the equation \eqref{otherboundary} that
 \beqn
 \Delta_\Omega^{(\ell)}(z,z) &=& \sigma_{\ell}(1,\ldots,1,t_{q+1}^2,\ldots,t_{r}^2). \eeqn
Thus, by Lemma \ref{numericalequation}, it follows that
\beqn
 \Delta_\Omega^{(\ell)}(z,z) &=& \sum_{i\geq \rm{max}(\ell-(r-q),0)}^{\rm{min}(q,\ell)} \binom{q}{i}\sigma_{\ell-i}(t_{q+1}^2,\ldots,t_r^2)\\
 &=& \sum_{i\geq \rm{max}(\ell-(r-q),0)}^{\rm{min}(q,\ell)} \binom{q}{i} \Delta_{\Omega'}^{(\ell-i)} \left(\sum_{i=q+1}^r t_i e_i', \sum_{i=q+1}^r t_i e_i'\right)\\
 &=& \sum_{i\geq \rm{max}(\ell-(r-q),0)}^{\rm{min}(q,\ell)} \binom{q}{i} \Delta_{\Omega'}^{(\ell-i)} (\Psi(z), \Psi(z)).
\eeqn

Conversely, assume that $z=k(\sum_{i=1}^r t_i e_i) \in \mathbb C^{r \times (r+b)}$ with $0 \leq t_r \leq \ldots \leq t_1$ such that \[\Delta_{\Omega}^{(\ell)}(z,z)= \sum_{i \geq \rm{max}(\ell-(r-q),0)}^{\rm{min}(q,\ell)} \binom{q}{i}\Delta_{\Omega'}^{(\ell-i)}(\Psi(z),\Psi(z)),\;\; \text{ for all } 1 \leq \ell \leq r.\]
Equivalently, \[\sigma_\ell(t_1^2,\ldots,t_r^2)= \sum_{i\geq \rm{max}(\ell-(r-q),0)}^{\rm{min}(q,\ell)} \binom{q}{i}\sigma_{\ell-i}(t_{q+1}^2,\ldots,t_r^2), \;\; \text{ for all } 1 \leq \ell \leq r.\] By Lemma \ref{numericalequation}, we have $t_1=\cdots=t_q=1,$ which shows that $z \in \overline{\partial \Omega_q}.$
\end{proof}
A potential notion of isometries associated with each boundary component, as established for the Shilov boundary in Theorem \ref{characterization}, can be considered. We plan to investigate this further in our future work.

%====================================================================
%\subsubsection*{Acknowledgment} 


\begin{thebibliography}{HD}
	
\bibitem{Al}
	A.B. Aleksandrov, \emph{ Inner functions on compact spaces}, Funct. Anal. Appl., {\bf 18} (1984), 87-98. 
	
\bibitem{Arazy} 
J. Arazy, \emph{A survey of invariant {H}ilbert spaces of analytic functions on bounded symmetric domains}, Contemporary Mathematics, {\bf 185}  (1995), 7-65.
\bibitem{At3}
A. Athavale, \emph{On the duals of subnormal tuples}, Integral Equations Operator Theory, {\bf 12} (1989), 305-323.
 
\bibitem{At2}
 \bysame, \emph{On the intertwining of joining isometries}, J. Operator Theory, {\bf 23} (1990), 339-350. 

\bibitem{At}
\bysame,  \emph{A note on Cartan isometries},
 New York J. Math.,  {\bf 25} (2019), 934–948.

\bibitem{Ber}
 H. Bercovici, \emph{A factorization theorem with applications to invariant subspaces and the reflexivity of isometries}, Math. Res. Lett., {\bf 1} (1994), 511-518.

 \bibitem{Tiret}
T. Bhattacharyya, B. K. Das, and H. Sau, \emph{Toeplitz Operators on the Symmetrized Bidisc}, Int. Math. Res. Not. IMRN, {\bf 11} (2021), 8492-8520.

\bibitem{Brown-Halmos}
A. Brown and P.R. Halmos, \emph{Algebraic properties of Toeplitz operators}, J. Reine Angew. Math., {\bf 213} (1963/64), 89-102.

\bibitem{Ca}
\'E. Cartan, \emph{ Sur les domaines born\'es homog\'enes de l'espace den variables complexes}, Abh. Math. Semin. Univ. Hambg., {\bf 11} (1935), 116-162.




\bibitem{CY}
S. Chavan and D. Yakubovich,
\emph{ Spherical tuples of Hilbert space operators},
Indiana Univ. Math. J. {\bf 64} (2015), 577-612.

\bibitem{Ch-Ze}
M. Chō, W. Żelazko, \emph{On geometric spectral radius radius of commuting $n$-tuples of operators}, Hokkaido Math. J.,  {\bf 21} (1992), 251–258. 

\bibitem{Conway}
J.B. Conway, \emph{A course in operator theory}, Grad. Stud. Math., {\bf 21}, American Mathematical Society, Rhode Island (2000).

\bibitem{DJ}
 A.M. Davie and N.P. Jewell, \emph{Toeplitz operator in several complex variables}, J. Funct. Anal., {\bf 26} (1977), 356-368.



\bibitem{Ding}
X. Ding, \emph{Products of Toeplitz operators in polydisc}, Integral Equations Operator Theory, {\bf 45} (2003), 389-403.

\bibitem{Ding-Sun-Zheng}
X. Ding, S. Sun, and D. Zheng, \emph{ Commuting Toeplitz operators on the bidisk}, J. Funct. Anal., {\bf 263} (2012), 3333-3357.

\bibitem{Did}
 M. Didas, \emph{Spherical isometries are reflexive }, Integral Equations Operator Theory, {\bf 52} (2005), 599-604.

 \bibitem{DE}
M. Didas and J. Eschmeier, \emph{ Inner functions and Spherical isometries}, Proc.  Amer. Math. Soc., {\bf 139} (2011), 2877-2889.  

\bibitem{DEE}
M. Didas, J. Eschmeier, and K. Everard, \emph{On the essential commutant of analytic Toeplitz operators associated with spherical isometries}, J. Funct. Anal., {\bf 261} (2011), 1361-1383.

\bibitem{DES}
M. Didas, J. Eschmeier, and D. Schillo, \emph{ On Scahtten-class perturbation of Toeplitz operators}, J. Funct. Anal., {\bf 272}(2017), 2442-2462.

\bibitem{Doug}
R. G. Douglas, \emph{Banach Algebra Techniques in Operator Theory}, second edition, Graduate Texts in Mathematics,  {\bf 179}, Springer-Verlag, New York, 1998.

\bibitem{Englis}
M. Engliš, \emph{Toeplitz operators and the Berezin transform on $H^2$}, Linear Algebra Appl.   {\bf 223/224} (1995), 171–204.

\bibitem{Esch1}
  J. Eschmeier, \emph{ On the reflexivity of multivariable isometries }, Proc. Amer. Math. Soc., {\bf 134} (2005), 1783-1789. 

  \bibitem{FK}
J. Faraut and A. Koranyi, \emph{ Function spaces and reproducing kernels on bounded symmetric domains}, J. Funct. Anal., {\bf 88} (1990), 64-89.

\bibitem{GKP}
	S. Ghara, S. Kumar, and P. Pramanick, \emph {$\mathbb K$-homogeneous tuple of operators on bounded symmetric domains}, Israel J. Math., {\bf 247} (2022), 331-360.

\bibitem{Ro-Go}
  G. Ghosh and S. S. Roy, \emph{Toeplitz operators on the proper images of bounded symmetric domains}, 
https://doi.org/10.48550/arXiv.2405.08002

 \bibitem{GR}
J. Gleason and S. Richter, \emph{m-Isometric Commuting Tuples of Operators on a Hilbert Space}, Integral Equations Operator Theory, {\bf 56} (2006), 181–196.

\bibitem{Gu}
C. Gu, \emph{Examples of $m$-isometric tuples of operators on a Hilbert space}, J. Korean Math. Soc.   {\bf 55} (2018), 225–251.

\bibitem{Subham-Par}
S. Jain and P. Pramanick, \emph{Toeplitz operators on the n-dimensional
Hartogs triangle}, to appear in J. Operator Theory (2024).

\bibitem{Koranyi}
A. Koran\'yi, \emph{The Poisson integral for generalized half-planes and bounded symmetric domains}, Ann. of Math. (2),   {\bf 82} (1965), 332–350.

\bibitem{K-W}
A. Koran\'yni and J. A. Wolf, \emph{Realization of hermitian symmetric spaces as generalized half-planes}, Ann. of Math. (2)   {\bf 81} (1965), 265–288.

\bibitem{Loos}
O. Loos, \emph{Bounded symmetric domains and Jordan pairs}, University of California, Irvine, 1977.

\bibitem{MM}
M. Mackey and P. Mellon, \emph{The Bergmann-Shilov boundary of a Bounded Symmetric Domain}, Math. Proc. R. Ir. Acad., {\bf 121A} ( 2021), 33-49.

\bibitem{MSS}
 A. Maji, J. Sarkar, and S. Sarkar, \emph{Toeplitz and asymptotic Toeplitz operators on $H^2(\mathbb{D}^n)$}, Bull. Sci. Math., {\bf 146} (2018), 33–49.

\bibitem{RN}
  R. Narasimhan, \emph{Several complex variables}, Chicago Lect. Math., The University of Chicago
  Press, (1971).

 \bibitem{BP}
 B. Prunaru, \emph{Some exact sequences for Toeplitz algebras of spherical isometries}, Proc.  Amer. Math. Soc., {\bf 135} (2007), 3621-3630.

 \bibitem{upmeier1}
 H. Upmeier, \emph{Toeplitz operators on bounded symmetric domains}, Trans. Amer. Math. Soc.,  {\bf 280} (1983), 221-237.

  \bibitem{upmeier2}
 \bysame, \emph{Toeplitz $C^*$-algebras on bounded symmetric domains}, Ann. of Math., {\bf 119} (1984), 549-576.
\bibitem{upmeier3}
H. Upmeier, \emph{ Toeplitz operators on symmetric Siegel domains}, Math. Ann., {\bf 271}(1985), 401–414.

 \bibitem{Hu}

\bysame, \emph{Jordan algebras and harmonic analysis on symmetric spaces}, Amer. J. Math., {\bf 108} (1986), 1-25.

\bibitem{upmeier}
\bysame, \emph{Toeplitz operators and index theory in several complex variables}, Operator Theory Advances and Applications,  {\bf 81}, Birkh\"auser, 1996.


\bibitem{Up}
 \bysame, \emph{Eigenvalues of $K$-invariant Toeplitz operators on bounded symmetric domains}, Integral Equations Operator Theory, {\bf 93}, (2021), Article no. 27. 

\bibitem{Weiss}
N. J. Weiss, \emph{Almost everywhere convergence of Poisson integrals on generalized half-planes}, Bull. Amer. Math. Soc.,   {\bf 74} (1968), 533–537.


 




\end{thebibliography}
\end{document}